\documentclass[english,twoside]{article}
\usepackage{cite}
\usepackage{amsmath,amssymb}
\usepackage{latexsym}
\usepackage[english]{babel}
\usepackage[latin1]{inputenc}
\usepackage{amsthm}
\usepackage{xcolor}
\usepackage{graphicx}

\newcommand{\N}{\mathbb{N}}
\newcommand{\Z}{\mathbb{Z}}

\newcommand{\R}{\mathbb{R}}

\newcommand{\A}{\mathbb{A}}
\newcommand{\D}{\mathbb{D}}
\newcommand{\Aa}{\check{\mathbb{A}}}
\newcommand{\s}{\mathbb{S}}

\newcommand{\tore}{\mathbb{T}}
\newcommand{\T}[1]{\tore^{#1}}

\newcommand{\abs}[1]{\left|#1\right|}

\newcommand{\fonc}[3]{#1:#2\to#3}
\newcommand{\sui}[1]{\left( #1_n \right)_{n\in\N}}
\newcommand{\suii}[2]{\left( #1_#2 \right)_{#2\in\N}}

\DeclareMathOperator{\rot}{Rot}
\DeclareMathOperator{\inte}{Int}
\DeclareMathOperator{\diam}{diam}
\DeclareMathOperator{\dom}{dom}

\DeclareMathOperator{\sing}{Sing}
\DeclareMathOperator{\fix}{Fix}

\theoremstyle{theorem}
\newtheorem{leem}{Lemma}[section]
\newtheorem{coro}[leem]{Corollary}
\newtheorem{prop}[leem]{Proposition}
\newtheorem{theo}[leem]{Theorem}

\newtheorem*{theoa}{Theorem A}

\newtheorem*{theob}{Theorem B}

\theoremstyle{remark}

\newtheorem{rema}{Remark}

\title{Existence of Non-Contractible Periodic Orbits for Homeomorphisms of the Open Annulus}
\author{\textsc{Jonathan Conejeros and F\'abio Armando Tal\footnote{The first author was supported by CNPq-Brasil. The second author was partially supported by Fapesp, CNPq-Brasil and CAPES.}}}
\date{}

\begin{document}
\maketitle

\begin{abstract}
  In this article we consider homeomorphisms of the open annulus $\A=\R/\Z\times \R$ which are isotopic to the identity and preserve a Borel probability measure of full support, focusing on the existence of non-contractible periodic orbits. Assume $f$ is such homeomorphism  such that the connected components of the set of fixed points of $f$ are all compact. Further assume that there exists $\check{f}$ a lift of $f$ to the universal covering of $\A$ such that the set of fixed points of $\check{f}$ is non-empty and that this set projects into an open topological disk of $\A$. We prove that, in this setting, one of the following two conditions must be satisfied: (1) $f$ has non-contractible periodic points of arbitrarily large periodic, or (2) for every compact set $K$ of $\A$ there exists a constant $M$ (depending on the compact set) such that, if $\check{z}$ and $\check{f}^n(\check{z})$ project on $K$, then their projections on the first coordinate have distance less or equal to $M$.
\end{abstract}

\section{Introduction}

H. Poincar\'e's  rotation number concept for circle homeomorphisms is one of the most celebrated and useful tools in dynamical systems theory, one that is familiar to almost all students and researchers in the field. Such success has lead to the generalization of the idea to several different contexts, where they have developed into important and useful tools. One of the first such generalization was the concept of rotation interval for endomorphisms of the circle, see \cite{palistakesn}. Given a continuous degree one map $f:\T{1}\to\T{1}$, where $\T{1}=\R/\Z$, and a lift $\check{f}:\R\to\R$, one can define the rotation set of $\check{f}$ as follows:
$$\rho(\check{f}):=\left\{\rho\, |  \,\exists\,\check{z}\in \R, \lim_{n\to +\infty} \frac{1}{n}(\check{f}^n(\check{z})-\check{z})=\rho  \right\},$$
and it can be shown that such set is always a closed interval.

As in the case of circle homeomorphisms, it is possible to show that $f$ has periodic points if and only if the rotation interval of $\check{f}$ contains a rational point and moreover, if $p/q\in\rho(\check{f})$, then there exists some $q$ periodic point $z$ with rotation number $p/q$, that is, there exists $\check{z}$ lifting $z$ such that $\check{f}^q(\check{z})= \check{z} + p$. Not only that, it can be shown that $\check f$  also presents {\it uniformly bounded deviations} from its rotation set, that is, there exists a positive constant $M$ such that, for all $\check z\in\R, n\in \Z$, it holds that $d(\check f^n(\check z)-\check z, n\rho(\check f))<M$. In particular, there exists a dichotomy between two different phenomena: A circle endomorphism can either have uniformly bounded deviations from a rigid rotation, or it must have a sequence of periodic points with arbitrarily large prime period and distinct rotation numbers.

As expected, there have been a large effort in trying to generalize the concept of rotation numbers for dynamical systems in dimension higher than one, trying to get suitable extensions of the $1$-dimensional results. But this task is not that simple, even for dimension 2, where the best attempts are only suitable for homeomorphisms in the homotopy class of the identity. The best known case here is the notion of rotation set of a torus homeomorphism, as introduced by M. Misiurewicz and K. Ziemian in \cite{mizi}. Given a homeomorphism $f:\T{2}\to\T{2}$ homotopic to the identity, where $\T{2}=\R^2/\Z^2$, and a lift $\check{f}:\R^2\to\R^2$ to its universal covering space and $\pi:\R^2\to\T{2}$ the covering projection, one can define the rotation set of $\check{f}$ as follows:
$$\rho(\check f):=\left\{v\, |\, \exists (n_k)_{k\in\N}\subset \N,\, \exists (\check{z}_k)_{k\in\N}\subset \R^2, \lim_{k\to +\infty} \frac{1}{n_k}(\check{f}^{n_k}(\check{z}_k)-\check{z}_k)= v \right\}.$$
Elements in $\rho(\check f)$ are called {\it rotation vectors} for $\check f$. Furthermore, if a point $z\in\T{2}$ is such that for some $\check z\in \pi^{-1}(z)$,  the limit $\lim_{n\to +\infty}\frac{\check f^n(\check z)-\check z}{n}=v$ exists, then $v$ is called the {\it rotation vector of $z$}.

The use of rotation sets as a tool for understanding and describing dynamical behavior in surfaces started in the 90s (see  \cite{franks1}, \cite{franks2}, \cite{mizi} and \cite{libreMcKay}), and has developed in a very active field in the last decade, with some significant contributions (see \cite{BoylandCarvalhoHall}, \cite{inventioneskt}). Lately, a couple of aspects of the subject have been drawing some increased attention. First, there have been several studies trying to determine how well does the rotation set capture the possible non-linear displacements. More specifically, under what conditions should one expect to get uniformly bounded deviations from the rotation set. To be precise, in the case where $f$ is a homeomorphism of the $2$-torus that is homotopic to the identity, and $\check f$ is a lift to the universal covering, when does it hold that there exists a positive constant $M$ such that, for all $\check z\in\R^2, n\in \N$, it holds that $d(\check f^n(\check z)-\check z, n\rho(\check f))<M$, as in the case of circle endomorphisms? It is known that this is not the case when the rotation set of $\check f$ is a singleton (see \cite{kocsardkkoro},\cite{ktexample}), even for area-preserving maps. On the other hand bounded deviations are present when the rotation set of $\check f$ has nonempty interior (see \cite{pablo}, \cite{salvador} and \cite{lct}). There has also been a large number of results with a similar flavor when the rotation set of $\check f$ is a non-degenerate line-segment, see for instance \cite{pablo}, \cite{kocsard} and \cite{koropasseguisambarino}.

A second direction, one that has drawn a particular interest due to connections to similar problems in symplectic dynamics, is to describe sufficient conditions for the existence of periodic points with distinct rotation numbers. Whenever $g$ is a homeomorphism of a manifold $M$ in the isotopy class of the identity and $\check g$ is a lift to the universal covering space $\check M$ commuting with the covering transformations, and $\check{\pi}:\check{M} \to M$ is the covering map, we say that a periodic point $z\in M$ is a {\it contractible periodic point} if every $\check z\in \check \pi^{-1}(z)$ is also periodic, otherwise we call $z$ a {\it non-contractible periodic point}. The question on whether a homeomorphism of $\T{2}$ has periodic points with distinct rotation numbers is often reduced to the study of the co-existence of {\it contractible} and {\it non-contractible} periodic points, and recent works concerning conditions for the existence of non-contractible periodic points in surface or symplectic dynamics can be found \cite{gurel}, \cite{gingurel} and \cite{tal}.

In this work we study how much of the bounded deviation machinery applies to the case of conservative maps on a non-compact surface. Our main goal is to examine how far the dichotomy that was present in the study of endomorphisms of the circle holds in the case of the open annulus $\A:=\T{1}\times\R$. We will denoted by $\check \pi:\Aa:=\R\times \R \to\A$ the universal covering map of $\A$. Our main result is the following.

\begin{theoa}\label{Theoa}
  Let $f$ be a homeomorphism of $\A$ which is isotopic to the identity and preserves a Borel probability measure of full support. Assume that the connected components of the set of fixed points of $f$ are all compact. Let $\check{f}$ be a lift of $f$ to $\Aa:=\R\times \R$. Assume that $\check{f}$ has fixed points and that there exists an open topological disk $U\subset \A$ such that the set of fixed points of $\check{f}$ projects into of $U$.  Then one of the following alternatives must hold:
  \begin{itemize}
    \item[(1)] there exists an integer $q\geq 1$ such that for every irreducible rational number $r/s\in (0,1/q]$ the map $\check{z}\mapsto \check{f}^s(\check{z})+(r,0)$ has a fixed point or for every irreducible rational number $r/s\in (0,1/q]$ the map $\check{z}\mapsto \check{f}^s(\check{z})-(r,0)$ has a fixed point. In particular, $f$ has non-contractible periodic points of arbitrarily large prime periodic.

    \item[(2)] for every compact set $K$ of $\A$ there exists a real constant $M>0$ such that for every point $\check{z}$ and every integer $n\geq 1$ such that $\check{z}$ and $\check{f}^n(\check{z})$ belong to $\check{\pi}^{-1}(K)$ one has
        $$ \abs{p_1(\check{f}^n(\check{z}))-p_1(\check{z})}\leq M,$$
  \end{itemize}
  where $p_1: \Aa\to\R$ is the projection on the first coordinate.
\end{theoa}

In order to prove Theorem A, we prove a recurrence type theorem in the lifted dynamics of a homeomorphism of $\A$ that is isotopic to the identity.

\begin{theob}\label{theob}
  Let $f$ be a homeomorphism of $\A$ which is isotopic to the identity and preserves a Borel probability measure of full support. Let $\check{f}$ be a lift of $f$ to $\Aa$. Suppose that $\check{f}$ has fixed points. Then one of the following alternatives must hold:
  \begin{itemize}
    \item [(1)] there exists an integer $q\geq 1$ such that for every irreducible rational number $r/s\in (0,1/q]$ the map $\check{z}\mapsto \check{f}^s(\check{z})+(r,0)$ has a fixed point or for every irreducible rational number $r/s\in (0,1/q]$ the map $\check{z}\mapsto \check{f}^s(\check{z})-(r,0)$ has a fixed point.

    \item [(2)] for every recurrent point $z\in\A$ there exists an open topological disk $V\subset \A$ containing $z$ such that: if $\check{V}$ is a lift of $V$ and $\check{z}$ is the lift of $z$ contained in $\check{V}$, then for every integer $n$ satisfying $f^n(z)\in V$ we have  $\check{f}^n(\check{z}) \in \check{V}$. In particular $\check{f}$ is non-wandering.
  \end{itemize}
\end{theob}
%
%
%

The main technique used in the proofs of Theorems A and B is the Equivariant Brouwer Theory of P. Le Calvez (see \cite{lecalvez1} and \cite{lecalvez2}) and a recently developed accompanying orbit forcing theory (see \cite{lct}). The paper is organized as follows: the second section introduces the basic lemmas and results from the above mentioned Equivariant Brouwer Theory and the forcing results, as well as details the concepts of rotation sets for annulus homeomorphisms and states some results that are used in the rest of the paper. Section 3 provides the necessary lemmas and results for obtaining Theorem B, Section 4 includes the proof of our main result and Section 5 provides two examples displaying how tight are the hypothesis of Theorem A.

{\it Acknowledgment.} We are very grateful for Patrice Le Calvez, whose several suggestions helped to improve the exposition of the paper and to greatly simplify some proofs.

\section{Preliminary results}

In this section, we state different results and definitions that will be useful in the rest of the article. The main tool will be the ``forcing theory'' introduced recently by P. Le Calvez and the second author (see \cite{lct} for more details). This theory will be expressed in terms of maximal isotopies, transverse foliations and transverse trajectories.

\subsection{The open annulus}
We will denote by $\T{1}$ the quotient space $\R/\Z$ and by $\A:=\T{1}\times \R$ the open annulus. We will endow $\A$ with its usual topology and orientation. We will denote by $\fonc{\check{\pi}}{\Aa=\R\times \R}{\A}$ the universal covering map of $\A$. We will denote by $\fonc{p_1}{\Aa}{\R}$ the projection on the first coordinate. We note that every nontrivial covering automorphism of $\check{\pi}$ is an iterate of the translation defined by $\check{z}\mapsto \check{z}+(1,0)$. A compact set $X\subset \A$ will be called \textit{essential} if its complement has two unbounded connected components, and a general set $X\subset \A$ will be called essential if it contains an essential compact set.

\subsection{Paths, lines, loops}
Let $M$ be an oriented surface. A \textit{path} on $M$ is a continuous map $\fonc{\gamma}{J}{M}$ defined on an interval $J$ of $\R$. In absence of ambiguity its image also will be called a path and denote by $\gamma$. We will denote $\fonc{\gamma^{-1}}{-J}{M}$ the path defined by $\gamma^{-1}(t)=\gamma(-t)$. If $X$ and $Y$ are two disjoint subsets of $M$, we will say that a path $\fonc{\gamma}{[a,b]}{M}$ \textit{joins} $X$ to $Y$ if $\gamma(a)\in X$ and $\gamma(b)\in Y$. A path $\fonc{\gamma}{J}{M}$ is \textit{proper} if the interval $J$ is open and the preimage of every compact subset of $M$ is compact. A \textit{line} on $M$ is an injective and proper path $\fonc{\lambda}{J}{M}$, it inherits a natural orientation induced by the usual orientation of $\R$. A path $\fonc{\gamma}{\R}{M}$ such that $\gamma(t+1)=\gamma(t)$ for every $t\in\R$ lifts a continuous map $\fonc{\Gamma}{\T{1}}{M}$. We will say that $\Gamma$ is a \textit{loop} and $\gamma$ is its natural lift. If $n\geq 1$ is an integer, we denote $\Gamma^n$ the loop lifted by the path $t\mapsto \gamma(nt)$.

\subsection{Lines of the plane}

Let $\lambda$ be a line of the plane $\R^2$. The complement of $\lambda$ has two connected components, $R(\lambda)$ which is on the right of $\lambda$ and $L(\lambda)$ which is on its left. If $X$ and $Y$ are two disjoint subsets of $\R^2$, we will say that a line $\lambda$ \textit{separates} $X$ from $Y$, if $X$ and $Y$ belong to different connected components of the complement of $\lambda$. Let us suppose that $\lambda_0$ and $\lambda_1$ are two disjoint lines of $\R^2$. We will say that $\lambda_0$ and $\lambda_1$ are \textit{comparable} if their right components are comparable for the inclusion. Note that $\lambda_0$ and $\lambda_1$ are not comparable if and only if $\lambda_0$ and $(\lambda_1)^{-1}$ are comparable.\\

Let us consider three lines $\lambda_0$, $\lambda_1$ and $\lambda_2$ in $\R^2$. We will say that $\lambda_2$ is \textit{above} $\lambda_1$ relative to $\lambda_0$ (and $\lambda_1$ is \textit{below} $\lambda_2$ relative to $\lambda_0$) if
\begin{itemize}
  \item the three lines are pairwise disjoint;
  \item none of the lines separates the two others;
    \item if $\gamma_1$ and $\gamma_2$ are two disjoint paths that join $z_1=\lambda_0(t_1)$, $z_2=\lambda_0(t_2)$ to $z'_1\in \lambda_1$, $z'_2\in \lambda_2$ respectively, and that do not meet the three lines but at the ends,
\end{itemize}
then $t_2>t_1$. This notion does not depend on the orientation of $\lambda_1$ and $\lambda_2$ but depends of the orientation of $\lambda_0$ (see Figure \ref{fig:orderoflines}).


\begin{center}
\begin{figure}[h!]
  \centering
    \includegraphics{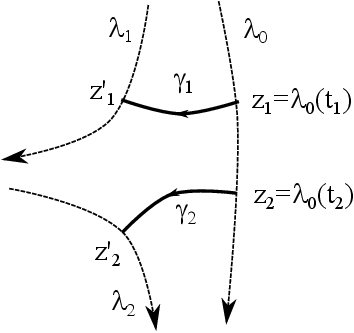}
  \caption{Order of lines. $\lambda_2$ is above $\lambda_1$ relative to $\lambda_0$ }
  \label{fig:orderoflines}
\end{figure}
\end{center}


\subsection{Transverse paths to surface foliations, transverse intersection}

Let $M$ be an oriented surface. By an \textit{oriented singular foliation} $\mathcal{F}$ on $M$ we mean a closed set $\sing (\mathcal{F})$, called \textit{the set of singularities of $\mathcal{F}$}, together with an oriented foliation $\mathcal{F}'$ on the complement of $\sing (\mathcal{F})$, called \textit{the domain of $\mathcal{F}$} denoted by $\dom(\mathcal{F})$, i.e. $\mathcal{F}'$ is a partition of $\dom(\mathcal{F})$ into connected oriented $1$-manifolds (circles or lines) called \textit{leaves of $\mathcal{F}$}, such that for every $z$ in $\dom(\mathcal{F})$ there exist an open neighborhood $W$ of $z$, called \textit{trivializing neighborhood} and an orientation-preserving homeomorphism called \textit{trivialization chart at $z$}, $\fonc{h}{W}{(0,1)^2}$ that sends the restricted foliation $\mathcal{F}|_W$ onto the vertical foliation oriented downward. If the singular set of $\mathcal{F}$ is empty, we will say that the foliation $\mathcal{F}$ is \textit{non singular}. For every $z\in \dom(\mathcal{F})$ we will write $\phi_z$ for the leaf of $\mathcal{F}$ that contains $z$. Let $\phi$ be a leaf of $\mathcal{F}$. Suppose that a point $z\in \phi$ has a trivialization neighborhood such that each leaf of $\mathcal{F}$ contains no more than one leaf of $\mathcal{F}|_W$. In this case every point of $\phi$ satisfies the same property. If furthermore no closed leaf of $\mathcal{F}$ meets $W$, we will say that $\phi$ is \textit{wandering}.\\


A path $\fonc{\gamma}{J}{\dom(\mathcal{F})}$ is \textit{positively transverse}\footnote{In the whole text ``transverse'' will mean ``positively transverse''.} to $\mathcal{F}$ if for every $t_0\in J$, and $h$ trivialization chart at $\gamma(t_0)$ the application $t\mapsto \pi_1(h(\gamma(t)))$, where $\fonc{\pi_1}{(0,1)^2}{(0,1)}$ is the projection on the first coordinate, is increasing in a neighborhood of $t_0$. We note that if $\check{M}$ is a covering space of $M$ and $\fonc{\check{\pi}}{\check{M}}{M}$ the covering projection, then $\mathcal{F}$ can be naturally lifted to a singular foliation $\check{\mathcal{F}}$ of $\check{M}$ such that $\dom(\check{\mathcal{F}})= \check{\pi}^{-1}(\dom(\mathcal{F}))$. We will denote by $\widetilde{\dom}(\mathcal{F})$ the universal covering space of $\dom(\mathcal{F})$ and $\widetilde{\mathcal{F}}$ the foliation lifted from $\mathcal{F}|_{\dom(\mathcal{F})}$. We note that $\widetilde{\mathcal{F}}$ is a non singular foliation of $\widetilde{\dom}(\mathcal{F})$.  Moreover if $\fonc{\gamma}{J}{\dom(\mathcal{F})}$ is \textit{positively transverse} to $\mathcal{F}$, every lift $\fonc{\check{\gamma}}{J}{\dom(\check{\mathcal{F}})}$ of $\gamma$ is \textit{positively transverse} to $\check{\mathcal{F}}$. In particular every lift $\fonc{\widetilde{\gamma}}{J}{\widetilde{\dom}(\mathcal{F})}$ of $\gamma$ to the universal covering space $\widetilde{\dom}(\mathcal{F})$ of $\dom(\mathcal{F})$  is \textit{positively transverse} to the lifted non singular foliation $\widetilde{\mathcal{F}}$.

\subsubsection{$\mathcal{F}$-transverse intersection for non singular planar foliations}

In this paragraph, we will suppose that $\mathcal{F}$ is a non singular foliation on the plane $\R^2$. We recall the following facts (see \cite{haereeb}).

\begin{itemize}
  \item Every leaf of $\mathcal{F}$ is a wandering line;
  \item the space of leaves of $\mathcal{F}$, denoted by $\Sigma$ furnished with the quotient topology, inherits a structure of connected and simply connected one-dimensional manifold;
  \item $\Sigma$ is Hausdorff if and only if the foliation $\mathcal{F}$ is trivial, that means that it is the image of the vertical foliation by a planar homeomorphism, or equivalently if all leaves of $\mathcal{F}$ are comparable.
\end{itemize}

    We will say that two transverse paths $\fonc{\gamma_1}{J_1}{\R^2}$ and $\fonc{\gamma_2}{J_2}{\R^2}$ are \textit{$\mathcal{F}$-equivalent} if they satisfy the three following equivalent conditions:

\begin{itemize}
  \item there exists an increasing homeomorphism $\fonc{h}{J_1}{J_2}$ such that for every $t\in J_1$ we have $\phi_{\gamma_1(t)}=\phi_{\gamma_2(h(t))}$;
  \item the paths $\gamma_1$ and $\gamma_2$ meet the same leaves of $\mathcal{F}$;
  \item the paths $\gamma_1$ and $\gamma_2$ project onto the same path of $\Sigma$.
\end{itemize}
Moreover, if $J_1=[a_1,b_1]$ and $J_2=[a_2,b_2]$ are two compact segments, these conditions are equivalent to next one:
\begin{itemize}
  \item One has $\phi_{\gamma_1(a_1)}= \phi_{\gamma_2(a_2)}$ and $\phi_{\gamma_1(b_1)}= \phi_{\gamma_2(b_2)}$.
\end{itemize}
In that case, note that the leaves met by $\gamma_1$ are the leaves $\phi$ of $\mathcal{F}$ such that $R(\phi_{\gamma_1(a_1)})\subset R(\phi) \subset R(\phi_{\gamma_2(b_1)})$. If the context is clear, we just say that the paths are equivalent and we omit the dependence on $\mathcal{F}$.\\

Let $\fonc{\gamma_1}{J_1}{\R^2}$ and $\fonc{\gamma_2}{J_2}{\R^2}$ be two transverse paths such that there are $t_1\in J_1$ and $t_2\in J_2$ satisfying $\phi_{\gamma_1(t_1)}=\phi_{\gamma_2(t_2)}=\phi$. We will say that $\gamma_1$ and $\gamma_2$ \textit{intersect $\mathcal{F}$-transversally and positively at $\phi$} if there exist $a_1,b_1$ in $J_1$ satisfying $a_1<t_1<b_1$, and $a_2,b_2$ in $J_2$ satisfying $a_2<t_2<b_2$, such that
\begin{itemize}
  \item $\phi_{\gamma_2(a_2)}$ is below $\phi_{\gamma_1(a_1)}$ relative to $\phi$; and
  \item $\phi_{\gamma_2(b_2)}$ is above $\phi_{\gamma_1(b_1)}$ relative to $\phi$.
\end{itemize}
See Figure \ref{fig:intersectiontransverse}.

\begin{center}
\begin{figure}[h!]
  \centering
    \includegraphics{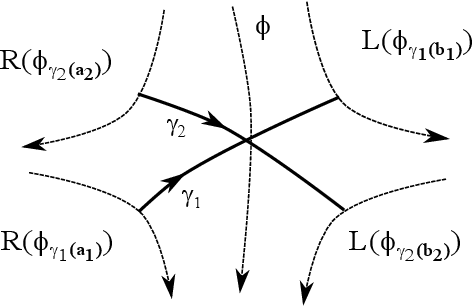}
  \caption{The paths $\gamma_1$ and $\gamma_2$ intersect $\mathcal{F}$-transversally and positively at $\phi$.}
  \label{fig:intersectiontransverse}
\end{figure}
\end{center}

Note that, if $\gamma_1$ intersects $\mathcal{F}$-transversally $\gamma_2$, if $\gamma'_1$ is equivalent to $\gamma_1$ and $\gamma'_2$ is equivalent to $\gamma_2$, then $\gamma'_1$ intersects $\mathcal{F}$-transversally $\gamma'_2$. We will say that the equivalence class of $\gamma_1$ intersects transversally the equivalence class of $\gamma_2$.

As none of the leaves $\phi$, $\phi_{\gamma_1(a_1)}$, $\phi_{\gamma_2(a_2)}$ separates the two others, one deduces that
$$  \phi_{\gamma_1(a_1)} \subset L(\phi_{\gamma_2(a_2)}),\quad  \phi_{\gamma_2(a_2)} \subset L(\phi_{\gamma_1(a_1)}).   $$
Similarly as none of the leaves $\phi$, $\phi_{\gamma_1(b_1)}$, $\phi_{\gamma_2(b_2)}$ separates the two others, one deduces that
$$  \phi_{\gamma_1(b_1)} \subset R(\phi_{\gamma_2(b_2)}), \quad  \phi_{\gamma_2(b_2)} \subset R(\phi_{\gamma_1(b_1)}).   $$

Observe that these properties remain true when $a_1$, $a_2$ are replaced by smaller parameters, $b_1$, $b_2$  by larger parameters, and $\phi$ by any another leaf met by $\gamma_1$ and $\gamma_2$. Note that $\gamma_1$ and $\gamma_2$ have at least one intersection point and that one can find two transverse paths $\gamma'_1$, $\gamma'_2$ equivalent to $\gamma_1$, $\gamma_2$ respectively, such that $\gamma'_1$ and $\gamma'_2$ have a unique intersection point, located in $\phi$, with a topological transverse intersection. Note that, if $\gamma_1$ and $\gamma_2$ are two paths that meet the same leaf $\phi$ of $\mathcal{F}$, then either they intersect $\mathcal{F}$-transversally, or one can find two transverse paths $\gamma'_1$, $\gamma'_2$ equivalent to $\gamma_1$, $\gamma_2$ respectively, with no intersection point.

\subsubsection{$\mathcal{F}$-intersection transverse in the general case}

Let us return now to the general case, i.e we will suppose that $\mathcal{F}$ is an oriented singular foliation on an oriented surface $M$. All previous definitions can be extended in the case that every connected component of $M$ is a plane and $\mathcal{F}$ is non singular. We will say that two transverse paths $\fonc{\gamma_1}{J_1}{\dom(\mathcal{F})}$ and $\fonc{\gamma_2}{J_2}{\dom(\mathcal{F})}$ are \textit{$\mathcal{F}$-equivalent} if there exist $\fonc{\widetilde{\gamma}_1}{J_1}{\widetilde{\dom}(\mathcal{F})}$ and $\fonc{\widetilde{\gamma}_2}{J_2}{\widetilde{\dom}(\mathcal{F})}$ lift of $\gamma_1$ and $\gamma_2$ respectively to the universal covering space  $\widetilde{\dom}(\mathcal{F})$ of $\dom(\mathcal{F})$ that are $\widetilde{\dom}(\mathcal{F})$-equivalent. This implies that there exists an increasing homeomorphism $\fonc{h}{J_1}{J_2}$ such that for every $t\in J_1$ we have $\phi_{\gamma_1(t)}=\phi_{\gamma_2(h(t))}$. Nevertheless these two conditions are not equivalent (see Figure 3 from\cite{lct}). One can prove that $\gamma_1$ and $\gamma_2$ are $\mathcal{F}$-equivalent if and only if, there exists a holonomic homotopy between $\gamma_1$ and $\gamma_2$, that is, if there exist a continuous transformation $\fonc{H}{J_1\times [0,1]}{\dom(\mathcal{F})}$ and an increasing homeomorphism $\fonc{h}{J_1}{J_2}$ satisfying:
\begin{itemize}
  \item $H(t,0)=\gamma_1(t)$, $H(t,1)=\gamma_2(h(t))$; and
  \item for every $t\in J_1$ and every $s_1,s_2\in [0,1]$, we have $\phi_{H(t,s_1)}= \phi_{H(t,s_2)}$.
\end{itemize}
We will say that a loop $\fonc{\Gamma}{\T{1}}{\dom(\mathcal{F})}$ is positively transverse to $\mathcal{F}$ if it is the case for its natural lift $\fonc{\gamma}{\R}{\dom(\mathcal{F})}$. We will say that two loops $\fonc{\Gamma_1}{J_1}{\dom(\mathcal{F})}$ and $\fonc{\Gamma_2}{J_2}{\dom(\mathcal{F})}$ are $\mathcal{F}$-equivalent if there exist two lifts $\fonc{\widetilde{\gamma}_1}{\R}{\widetilde{\dom}(\mathcal{F})}$ and $\fonc{\widetilde{\gamma}_2}{\R}{\widetilde{\dom}(\mathcal{F})}$  of $\Gamma_1$ and $\Gamma_2$ respectively to the universal covering space  $\widetilde{\dom}(\mathcal{F})$ of $\dom(\mathcal{F})$, a covering automorphism $T$ and an orientation preserving homeomorphism $\fonc{h}{\R}{\R}$ such that for every $t\in \R$ we have
$$ \widetilde{\gamma}_1(t+1)=T(\widetilde{\gamma}_1(t)), \; \widetilde{\gamma}_2(t+1)=T(\widetilde{\gamma}_2(t)),\; h(t+1)=h(t)+1, \; \phi_{\widetilde{\gamma}_2(h(t))}=\phi_{\widetilde{\gamma}_1(t)}.$$ We note that for every integer $n\geq 1 $ the loops $\Gamma^n_1$ and $\Gamma^n_2$ are $\mathcal{F}$-equivalent if this is the case for $\Gamma_1$ and $\Gamma_2$. If two loops $\Gamma_1$ and $\Gamma_2$ are $\mathcal{F}$-equivalent, then there exists a holonomic homotopy between them and therefore they are freely homotopic in $\dom(\mathcal{F})$. Nevertheless the converse does not hold (see Figure 4 from \cite{lct}).\\

 Let $\fonc{\gamma_1}{J_1}{M}$ and $\fonc{\gamma_2}{J_2}{M}$ be two transverse paths that meet a common leaf $\phi=\phi_{\gamma_1(t_1)}=\phi_{\gamma_2(t_2)}$. We will say that $\gamma_1$ and $\gamma_2$ \textit{intersect $\mathcal{F}$-transversally at $\phi$} if there exist paths $\fonc{\widetilde{\gamma}_1}{J_1}{\widetilde{\dom}(\mathcal{F})}$ and $\fonc{\widetilde{\gamma}_2}{J_2}{\widetilde{\dom}(\mathcal{F})}$, lifting $\gamma_1$ and $\gamma_2$, with a common leaf $\widetilde{\phi}=\phi_{\widetilde{\gamma}_1(t_1)}=\phi_{\widetilde{\gamma}_2(t_2)}$ that lifts $\phi$, and intersecting $\widetilde{\mathcal{F}}$-transversally at $\widetilde{\phi}$. Here again, we can give sign to the intersection. As explained in the last subsection, there exist $t'_1\in J_1$ and $t'_2\in J_2$ such that $\gamma_1(t'_1)=\gamma_2(t'_2)$ and such that $\gamma_1$ and $\gamma_2$ intersect $\mathcal{F}$-transversally at $\phi_{\gamma_1(t'_1)}=\phi_{\gamma_2(t'_2)}$. In this case we will say that $\gamma_1$ and $\gamma_2$ intersect $\mathcal{F}$-transversally at $\gamma_1(t'_1)=\gamma_2(t'_2)$. In the case where $\gamma_1=\gamma_2$ we will talk of a $\mathcal{F}$-transverse self-intersection. A transverse path $\gamma$ has a $\mathcal{F}$-transverse self-intersection if for every lift $\widetilde{\gamma}$ to the universal covering space of the domain of the foliation, there exists a non trivial covering automorphism $T$ such that $\widetilde{\gamma}$ and $T(\widetilde{\gamma})$ have a $\widetilde{\mathcal{F}}$-transverse intersection. \\

Similarly, let $\Gamma$ be a loop that is transverse to $\mathcal{F}$ and $\gamma$ its natural lift. If $\gamma$ intersects $\mathcal{F}$-transversally a transverse path $\gamma'$ at a leaf $\phi$, we will say that $\Gamma$ and $\gamma'$ intersect $\mathcal{F}$-transversally at $\phi$. Moreover if $\gamma'$ is the natural lift of a transverse loop $\Gamma'$, we will say that $\Gamma$ and $\Gamma'$ intersect $\mathcal{F}$-transversally at $\phi$. Here again we can talk of self-intersection.\\

A transverse path $\fonc{\gamma}{\R}{M}$ will be called \textit{$\mathcal{F}$-positively recurrent} if for every segment $J$ of $\R$ and every $t\in\R$ there exists a segment $J'$ contained in $[t,+\infty)$ such that $\gamma|_{J'}$ is $\mathcal{F}$-equivalent to $\gamma|_J$. It will be called \textit{$\mathcal{F}$-negatively recurrent} if for every segment $J$ of $\R$ and every $t\in\R$ there exists a segment $J'$ contained in $(-\infty,t]$ such that $\gamma|_{J'}$ is $\mathcal{F}$-equivalent to $\gamma|_J$. It will be called \textit{$\mathcal{F}$-bi-recurrent} if it is both $\mathcal{F}$-positively and $\mathcal{F}$-negatively recurrent. We note that, if $\fonc{\gamma}{\R}{M}$ and $\fonc{\gamma'}{\R}{M}$ are $\mathcal{F}$-equivalent and if $\gamma$ is $\mathcal{F}$-positively recurrent (or $\mathcal{F}$-negatively recurrent), then so is $\gamma'$. We will say that a $\mathcal{F}$-equivalent class is positively recurrent (negatively recurrent or bi-recurrent) is some representative of the class is $\mathcal{F}$-positively recurrent (respectively $\mathcal{F}$-negatively recurrent, $\mathcal{F}$-bi-recurrent).

The following result describes paths with no transverse self-intersection on the two-dimensional sphere.

\begin{prop}[\cite{lct}]\label{proposition2LCT}
  Let $\mathcal{F}$ be an oriented singular foliation on $\s^2$ and let $\fonc{\gamma}{\R}{\s^2}$ be a $\mathcal{F}$-bi-recurrent transverse path. The following properties are equivalent:
  \begin{itemize}
    \item[(i)] $\gamma$ has no $\mathcal{F}$-transverse self-intersection;
    \item[(ii)] there exists a transverse simple loop $\Gamma'$ such that $\gamma$ is equivalent to the natural lift $\gamma'$ of $\Gamma'$;
    \item[(iii)] the set $U=\bigcup_{t\in\R}\phi_{\gamma(t)}$ is an open annulus.
  \end{itemize}
\end{prop}
%
%

\subsection{Maximal isotopies, transverse foliations, admissible paths, realizability of linearly admissible transverse loops}

\subsubsection{Isotopies, maximal isotopies}

Let $M$ be an oriented surface. Let $f$ be a homeomorphism of $M$. An \textit{identity isotopy} of $f$ is a path that joins the identity to $f$ in the space of homeomorphisms of $M$, furnished with the $C^0$-topology. We will say that \textit{$f$ is isotopic to the identity} if the set of identity isotopies of $f$ is not empty. Let $I=\left(f_t\right)_{t\in [0,1]}$ be an identity isotopy of $f$. Given $z\in M$ we can define the \textit{trajectory of $z$} as the path $I(z): t \mapsto f_t(z)$. More generally, for every integer $n\geq 1$ we define $I^n(z)=\prod_{0\leq k<n} I(f^k(z))$ by concatenation. We will also use the following notations

$$ I^{\N}(z)= \prod_{k\in \N} I(f^k(z)), \quad I^{-\N}(z)= \prod_{k\in \N} I(f^{-k}(z)), \quad I^{\Z}(z)= \prod_{k\in \Z} I(f^k(z)).     $$
The last path will be called the \textit{whole trajectory of $z$}. One can define the fixed point set of $I$ as $\fix(I)= \cap_{t\in [0,1]} \fix(f_t)$, which is the set of points with trivial whole trajectory. The complement of $\fix(I)$ will called the \textit{domain of $I$}, and it will be denoted by $\dom(I)$.\\
In general, let us say that an identity isotopy of $f$ is a maximal isotopy, if there is no fixed point of $f$ whose trajectory is contractible relative to the fixed point set of $I$. A very recent result of F. B\'eguin, S. Crovisier and F. Le Roux (see \cite{BCLR16}) asserts that such an isotopy always exists if $f$ is isotopic to the identity (a slightly weaker result was previously proved by O. Jaulent (see \cite{jaulent})). Here we prefer to follow \cite{BCLR16}, because Jaulent's Theorem about existence of maximal isotopies cannot be stated in the following natural form.

\begin{theo}[\cite{jaulent}, \cite{BCLR16}]\label{existence maximal isotopy}
  Let $M$ be an oriented surface. Let $f$ be a homeomorphism of $M$ which is isotopic to the identity and let $I'$ be an identity isotopy of $f$. Then there exists an identity isotopy $I$ of $f$ such that:
  \begin{itemize}
    \item[(i)] $\fix(I')\subset \fix(I)$;
    \item[(ii)] $I$ is homotopic to $I'$ relative of $\fix(I')$;
    \item[(iii)] there is no point $z\in \fix(f)\setminus \fix(I)$ whose trajectory $I(z)$ is homotopic to zero in $M\setminus \fix(I)$.
  \end{itemize}
\end{theo}
We will say that an identity isotopy $I$ satisfying the conclusion of Theorem \ref{existence maximal isotopy} is a \textit{maximal isotopy}. We note that the last condition of the above theorem can be stated in the following equivalent form:
 \begin{itemize}
    \item[(iii')] if $\widetilde{I}=(\widetilde{f}_t)_{t\in [0,1]}$ is the identity isotopy that lifts $I|_{M\setminus \fix(I)}$ to the universal covering space of $M\setminus \fix(I)$, then $\widetilde{f}_1$ is fixed point free.
  \end{itemize}
The typical example of an isotopy $I$ verifying condition (iii) is the restricted family $I=(f_t)_{t\in [0,1]}$ of a topological flow $(f_t)_{t\in \R}$ on $M$. Indeed, one can lift the flow $(f_t| _{M\setminus \fix(I)})_{t\in \R}$ as a flow $(\widetilde{f}_t)_{t\in \R}$ on the universal covering space of $M\setminus \fix(I)$. This flow has no fixed point and consequently no periodic point. So $\widetilde{f}_1$ is fixed point free, which exactly means that condition (iii) is fulfilled.

\subsubsection{Transversal foliations}

Let $f$ be a homeomorphism that preserves the orientation of the plane $\R^2$. We will say that a line $\lambda$ of $\R^2$ is a \textit{Brouwer line of $f$} if $f(\lambda)\subset L(\lambda)$ and $f^{-1}(\lambda)\subset R(\lambda)$. If $f$ is fixed point free the main result to the Brouwer Theory is the Plane Translation Theorem: every point of the plane lies on a Brouwer line of $f$ (see \cite{brou}). Let us recall now the equivariant foliation version of this theorem due to P. Le Calvez (see \cite{lecalvez2}). Suppose that $f$ is a homeomorphism that is isotopic to the identity on an oriented surface $M$. Let $I$ be a maximal identity isotopy of $f$ and let $\widetilde{I}=(\widetilde{f}_t)_{t\in [0,1]}$ be the identity isotopy that lifts $I$ to the universal covering space $\widetilde{\dom}(I)$ of $\dom(I)$. We recall that the homeomorphism $\widetilde{f}=\widetilde{f}_1$ is fixed point free. Suppose that $\dom(I)$ is connected, in this case $\widetilde{\dom}(I)$ is a plane and we have that there exists a non singular oriented foliation $\widetilde{\mathcal{F}}$ on $\widetilde{\dom(I)}$, invariant by the covering automorphisms, whose leaves are Brouwer lines of $\widetilde{f}$ (see \cite{lecalvez2}). We have the following result, still true in case that $\dom(I)$ is not connected.

\begin{theo}[\cite{lecalvez2}]\label{existence transverse foliation}
  Let $M$ be an oriented surface. Let $f$ be a homeomorphism of $M$ which is isotopic to the identity and let $I$ be a maximal identity isotopy of $f$. Then there exists an oriented singular foliation $\mathcal{F}$ with $\dom(\mathcal{F})=\dom(I)$, such that for every $z\in \dom(I)$ the trajectory $I(z)$ is homotopic, relative to the endpoints, to a positively transverse path to $\mathcal{F}$ and this path is unique defined up to equivalence.
\end{theo}
We will say that a foliation $\mathcal{F}$ satisfying the conclusion of Theorem \ref{existence transverse foliation} is \textit{transverse} to $I$. Observe that if $\check{M}$ is a covering space of $M$ and $\fonc{\check{\pi}}{\check{M}}{M}$ the covering projection, a foliation $\mathcal{F}$ transverse to a maximal identity isotopy $I$ lifts to a foliation $\check{\mathcal{F}}$ transverse to the lifted isotopy $\check{I}$.\\

Given $z\in M$ we will write $I_{\mathcal{F}}(z)$ for the class of paths that are positively transverse to $\mathcal{F}$, that join $z$ to $f(z)$ and that are homotopic in $\dom(\mathcal{F})$ to $I(z)$, relative to the endpoints. We will also use the notation $I_{\mathcal{F}}(z)$ for every path in this class and we will called it the transverse trajectory of $z$. More generally, for every integer $n\geq 1$ we can define $I^n_{\mathcal{F}}(z)=\prod_{0\leq k<n} I(f^k(z))$ by concatenation, that is either a transverse path passing through the points $z$, $f(z)$, $\cdots$, $f^n(z)$, or a set of such paths. We will also use the following notations

$$ I^{\N}_{\mathcal{F}}(z)= \prod_{k\in \N} I(f^k(z)), \quad I^{-\N}_{\mathcal{F}}(z)= \prod_{k\in \N} I(f^{-k}(z)), \quad I^{\Z}_{\mathcal{F}}(z)= \prod_{k\in \Z} I(f^k(z)). $$
The last path will be called the \textit{whole transverse trajectory of $z$}.\\

If $z$ is a periodic point of $f$ of period $q$, there exists a transverse loop $\Gamma$ whose natural lift $\gamma$ satisfies $\gamma|_{[0,1]}=I_{\mathcal{F}}^q(z)$. We will say that a transverse loop \textit{is associated to $z$} if it is $\mathcal{F}$-equivalent to $\Gamma$. We note that this definition does not depend on the choices of the trajectory $I_{\mathcal{F}}(f^k(z))$, with $0\leq k<q$.\\

 Let us state the following result that will be useful later.

\begin{leem}[\cite{lct}]\label{lemma10LCT}
  Fix $z\in \dom(I)$, an integer $n\geq 1$, and parameterize $I^n_{\mathcal{F}}(z)$ by $[0,1]$. For every $0<a<b<1$, there exists a neighborhood $V$ of $z$ such that for every $z'$ in $V$, the path $I^n_{\mathcal{F}}(z)|_{[a,b]}$ is equivalent to a subpath of $I^n_{\mathcal{F}}(z')$. Moreover, there exists a neighborhood $W$ of $z$ such that for every $z'$ and $z''$ in $W$, the path $I^n_{\mathcal{F}}(z')$ is equivalent to a subpath of $I^{n+2}_{\mathcal{F}}(f^{-1}(z''))$.
\end{leem}

An immediate consequence of the previous lemma is the fact that if $z$ in $\dom(I)$ is positively recurrent, negatively recurrent or bi-recurrent, then the whole transverse trajectory of $z$, $I^{\Z}_{\mathcal{F}}(z)$ is $\mathcal{F}$-positively recurrent, $\mathcal{F}$-negatively recurrent or $\mathcal{F}$-bi-recurrent respectively.

\subsubsection{Admissible paths}

We will say that a transverse path $\fonc{\gamma}{[a,b]}{\dom(I)}$  is \textit{admissible of order $n$} ($n$ is an integer larger than $1$) if it is equivalent to a path $I^{n}_{\mathcal{F}}(z)$, $z$ in $\dom(I)$. It means that if $\fonc{\widetilde{\gamma}}{[a,b]}{\widetilde{\dom}(I)}$ is a lift of $\gamma$, then there exists a point $\widetilde{z}$ in $\widetilde{\dom}(I)$ such that $\widetilde{z}\in \phi_{\widetilde{\gamma}(a)}$ and $\widetilde{f}^n(\widetilde{z})\in \phi_{\widetilde{\gamma}(b)}$, or equivalently, that
$$    \widetilde{f}^n(\phi_{\widetilde{\gamma}(a)})\cap \phi_{\widetilde{\gamma}(b)} \neq \emptyset.  $$

We note that if $f$ preserves a Borel probability measure of full support, then the set of bi-recurrent points is dense in $M$. It follows from Lemma \ref{lemma10LCT} that every admissible transverse path is equivalent to a subpath of a bi-recurrent one.\\

We will say that a transverse path $\fonc{\gamma}{[a,b]}{\dom(I)}$ is \textit{admissible of order $\leq n$} ($n$ is an integer larger than $1$) if it is a subpath of an admissible path of order $n$. More generally, we will say that a transverse path $\fonc{\gamma}{J}{\dom(I)}$ defined on an interval of $\R$ is \textit{admissible} if for every segment $[a,b]\subset J$, there exists an integer $n\geq 1$ such that $\gamma|_{[a,b]}$ is admissible of order $\leq n$. Similarly, we will say that a transverse loop $\Gamma$ is admissible if its natural lift is admissible. The following lemma follows from Proposition 19 of \cite{lct} and states that for transverse paths with a $\mathcal{F}$-transverse self-intersection there is no difference between being of order $\leq n$ and being of order $n$.

\begin{leem}\label{Proposition19LCT}
  Let $\fonc{\gamma}{[a,b]}{\dom(I)}$ be a transverse path with a $\mathcal{F}$-transverse self-intersection. If $\gamma$ is admissible of order $\leq n$, then $\gamma$ is admissible of order $n$.
\end{leem}

The fundamental proposition (Proposition 20 from \cite{lct}) is a result about maximal isotopies and transverse foliations that permits us to construct new admissible paths from a pair of admissible paths.

\begin{prop}[\cite{lct}]\label{proposition20LCT}
  Suppose that $\fonc{\gamma_1}{[a_1,b_1]}{M}$ and $\fonc{\gamma_2}{[a_2,b_2]}{M}$ are two transverse paths that intersect $\mathcal{F}$-transversally at
  $\gamma_1(t_1)=\gamma_2(t_2)$. If $\gamma_1$ is admissible of order $n_1$ and $\gamma_2$ is admissible of order $n_2$, then the paths
  $\gamma_1|_{[a_1,t_1]}\gamma_2|_{[t_2,b_2]}$ and $\gamma_2|_{[a_2,t_2]}\gamma_1|_{[t_1,b_1]}$ are admissible of order $n_1+n_2$.
\end{prop}

One deduces immediately the following result, Corollary 22 of \cite{lct}.

\begin{leem}[\cite{lct}]\label{corollary22LCT}
  Let $\fonc{\gamma_i}{[a_i,b_i]}{M}$, $1\leq i \leq r$, be a family of $r\geq 2$ transverse paths. We suppose that for every $i\in \{1,\cdots, r\}$ there exist $s_i\in [a_i,b_i]$ and $t_i\in [a_i,b_i]$ such that:
  \begin{itemize}
    \item[(i)] $\gamma_i|_{[s_i, b_i]}$ and $\gamma_{i+1}|_{[a_{i+1},t_{i+1}]}$ intersect $\mathcal{F}$-transversally and positively at $\gamma_i(t_i)=\gamma_{i+1}(s_{i+1})$ if $i<r$;
    \item[(ii)] one has $s_1=a_1<t_1<b_1$, $a_r<s_r<t_r=b_r$ and $a_i<s_i<t_i<b_i$ if $1<i<r$;
    \item[(iii)] $\gamma_i$ is admissible of order $n_i$.
  \end{itemize}
Then $\prod_{1\leq i \leq r} \gamma_i|_{[s_i,t_i]}$ is admissible of order $\sum_{1\leq i \leq r} n_i$.
\end{leem}

The following result is a consequence of Proposition 23 from \cite{lct}.

\begin{leem}[\cite{lct}]\label{corollary24LCT}
  Let $\fonc{\gamma}{[a,b]}{M}$ be a transverse path admissible of order $n$. Then there exists $\fonc{\gamma'}{[a,b]}{M}$ a transverse path, also admissible of order $n$, such that $\gamma'$ has no $\mathcal{F}$-transverse self-intersection and $\phi_{\gamma(a)}=\phi_{\gamma'(a)}$, $\phi_{\gamma(b)}=\phi_{\gamma'(b)}$.
\end{leem}

\subsubsection{Realizability of linearly admissible loops}

Let $\Gamma$ be a $\mathcal{F}$-transverse loop and let $\fonc{\gamma}{\R}{M}$ be its natural lift. We will say that $\Gamma$ is \textit{linearly admissible of order $q$} ($q$ is an integer larger than $1$) if it satisfies the following property (note that every equivalent loop will satisfy the same property): \\
$(Q_q)$ : there exist two sequences $\suii{r}{k}$ and $\suii{s}{k}$ of natural integers satisfying
$$  \lim_{k\to +\infty} r_k = \lim_{k\to +\infty} s_k=+\infty, \quad  \limsup_{k\to +\infty} r_k/s_k \geq 1/q          $$
such that for every integer $k\geq 1$,  $\gamma|_{[0,r_k]}$ is admissible of order $\leq s_k$.

In \cite{lct} the authors proved that in many situations the existence of a transverse loop that satisfies property $(Q_q)$ implies the existence of infinitely many periodic orbit. In our setting we can state their result as follows.

\begin{prop}[Proposition 26 of \cite{lct}]\label{proposition26LCT}
  Let $\Gamma$ be a linearly admissible transverse loop of order $q\geq 1$ that has a $\mathcal{F}$-transverse self-intersection. Then for every rational number $r/s\in (0,1/q]$, written in an irreducible way, the loop $\Gamma^r$ is associated to a periodic orbit of period $s$.
\end{prop}

The following result will permit to apply the previous proposition.

\begin{leem}[Lemma 30 of \cite{lct}]\label{lemma30LCT}
  Let $\fonc{\gamma_1,\gamma_2}{\R}{M}$ be two admissible positively recurrent paths (possibly equal) with a $\mathcal{F}$-transverse intersection, and let $I_1$ and $I_2$ be two real segments. Then there exists a linearly admissible transverse loop $\Gamma$ with a $\mathcal{F}$-transverse self-intersection, such that $\gamma_1|_{I_1}$ and $\gamma_2|_{I_2}$ are equivalent to subpaths of the natural lift of $\Gamma$.
\end{leem}

\subsection{Rotation set of annular homeomorphisms}\label{sectionrotationset}

In this paragraph, we consider a homeomorphism $f$ of the open annulus $\A=\T{1}\times \R$ which is isotopy to the identity. Let $\check{f}$ be a lift of $f$ to $\check{\A}$. We will give the definition of the \textit{rotation set of $\check{f}$} due to J. Franks (see \cite{franks96}). Given a compact set $K$ of $\A$, a number $\rho\in \overline{\R}:=\R\cup \{+\infty\}\cup \{-\infty\}$ belongs to the \textit{rotation set of $\check{f}$ relative to $K$}, denoted by $\rot_{K}(\check{f})$ if there exist a sequence of points $(\check{z}_k)_{k\in\N}$ and a sequence of integers $(n_k)_{k\in\N}$ which goes to $+\infty$ such that for every $k\in\N$, $\check{z}_k$ and $\check{f}^{n_k}(\check{z}_k)$ belong to $\check{\pi}^{-1}(K)$ and
    $$  \rho= \lim_{k\to +\infty} \frac{1}{n_k}(p_1(\check{f}^{n_k}(\check{z}_k))-p_1(\check{z}_k)). $$
The \textit{rotation set of $\check{f}$} is defined as
$$ \rot(\check{f}):=\overline{\bigcup_{K} \rot_{K}(\check{f})},$$
where $K$ is a compact set of $\A$ and the ``closure'' is taken in $\overline{\R}$.
We say that a point {\em $z\in \A$ has rotation number equal to $\rho$} if for any lift $\check{z}$ of $z$, we have that the limit $\lim_{n\to +\infty} \frac{1}{n}(p_1(\check{f}^{n}(\check{z}))-p_1(\check{z}))$ exists and it is equal to $\rho$. Note that if the limit exists, it is independent of $\check{z}\in \check{\pi}^{-1}(z)$.
We note that for every $p\in\Z$ and every $q\in\Z$, the map $\check{f}^q +(p,0)$ defined by $\check{z}\mapsto \check{f}^q(\check{z}) +(p,0)$ is a lift of $f^q$ and we have $\rot(\check{f}^q+ (p,0))=q\rot(\check{f})+p$. We note that the first author proved that this set is always an interval (see \cite{con}). We recall that this result has been known for measure-preserving homeomorphisms (for example, see \cite{lecalvez2}, Theorem 9.1 for a proof that uses maximal isotopies and transverse foliations). In this case, we obtain the following theorem.

\begin{theo}[\cite{franks96}, \cite{lecalvez2}]\label{rotationset is a interval}
  Let $f$ be a homeomorphism of $\A$ which is isotopic to the identity and preserves a Borel probability measure of full support. Let $\check{f}$ be a lift of $f$ to $\Aa$. Then for every irreducible rational number $r/s$ that belongs to the interior of $\rot(\check{f})$ there exists a point $\check{z}$ in $\check{\A}$ such that $\check{f}^s(\check{z})= \check{z}+(r,0)$.
\end{theo}

%
%

\subsection{A classical Brouwer theory lemma}

In this section, we will prove a general proposition that play a key role in the proof of the existence of non-contractible periodic point (see subsection \ref{corolemmaintersectiontransverse1}). We will use classical properties of translations, derived from Brouwer Theory. The following lemma is a direct consequence of Lemma 3.1 of \cite{brown}.

\begin{leem}\label{lemma3GKT}
  Let $K\subset \R^2$ be an arcwise connected set such that $K\cap (K+(1,0))=\emptyset$. Then for every $j\in \Z$ with $j\neq 0$, we have that $K\cap (K+(j,0))=\emptyset$.
\end{leem}

The following proposition is Lemma 12 from \cite{gukotal}. We outline here the proof.

\begin{prop}[\cite{gukotal}]\label{propocle}
  Let $\delta$ be a segment such that $\delta\cap (\delta+(1,0))=\emptyset$. Let $\fonc{\gamma}{[0,1]}{\R^2}$ be a path satisfying $\gamma(0)\in \delta$ and $\gamma(1)\in \delta+(j,0)$ for some $j\in\N$. Then
  \begin{itemize}
    \item[(i)] the path $\gamma$ meets $\gamma+(1,0)$, or
    \item[(ii)] for every $i\in \{0,\cdots,j\}$, the path $\gamma$ meets $\delta+(i,0)$.
  \end{itemize}
\end{prop}
\begin{proof}
  Suppose that Assertion $(i)$ does not hold, that is $\gamma\cap (\gamma+(1,0))=\emptyset$. Fix $i\in \{0,\cdots,j\}$, we will prove that $\gamma\cap (\delta+(i,0))\neq \emptyset$.  Consider
  $$  t_0:=\max\left\{ t\in [0,1] : \gamma(t)\in \bigcup_{n\geq 0} (\delta+(i-n,0))  \right\},   $$
  $$  t_1:=\min\left\{ t\in [t_0,1] : \gamma(t)\in \bigcup_{n\geq 1} (\delta+(i+n,0))  \right\},   $$
and let $i_0\geq 0$ and $j_0\geq  1 $ be integers such that $\gamma(t_0)\in \delta+(i-i_0,0)$ and $\gamma(t_1)\in \delta+(i+j_0,0)$. Finally, let
$$  K:= (\delta+(i-i_0,0))\cup \gamma([t_0,t_1])\cup (\delta+(i+j_0,0)), $$ which is an arcwise connected set.
We prove by contradiction that $i_0=0$ and $j_0=1$, i.e. we assume that $i_0+j_0>1$ and we seek a contradiction. Note first that by construction, the path $\gamma((t_0,t_1))$ is disjoint from $\cup_{n\in\Z} (\delta+(n,0))$.  Since $\delta\cap (\delta+(1,0))=\emptyset$, it follows from Lemma \ref{lemma3GKT} that for every non-zero integer $n$ we have $\delta\cap (\delta+(n,0))=\emptyset$, and so $\delta+(i-i_0,0)$ is disjoint from $\delta+(i+j_0,0)$ (because $i_0+j_0>1$). Therefore we have that $K$ is disjoint from $K+(1,0)$, and so again by Lemma \ref{lemma3GKT} we have that for every non-zero integer $n$, $K$ is disjoint from $K+(n,0)$. But $i_0+j_0>1$, and clearly $K+(i_0+j_0,0)$ intersects $K$. This contradiction shows that $i_0+j_0=1$, i.e. $i_0=0$ and $j_0=1$.\\ Since
$\gamma(t_0)\in \gamma\cap (\delta+(i-i_0,0))= \gamma\cap (\delta+(i,0))$, we have shown that $\gamma$ intersects $\delta+(i,0)$, i.e. Assertion $(ii)$ holds. This completes the proof of the proposition.
\end{proof}

\subsection{A triple boundary lemma for surface homeomorphisms}

In this section we will state a theorem recently proved by A. Koropecki, P. Le Calvez and the second author (see \cite{KLCT}). It plays a key role in the proof of Theorem A because simplifies the first version of its proof. We denote by $\s^2$ the sphere of dimension 2.

\begin{theo}\label{a triple boundary lemma}
Suppose that $f:\s^2\to \s^2$ is an orientation-preserving homeomorphism, and $B$ is a closed topological disk such that $f(B)\cap B=\emptyset$. If $B$ intersects three pairwise disjoint open $f$-invariant topological disks, then $f$ has wandering points.
\end{theo}

A point $x$ is called  \textit{wandering} for a homeomorphism $f$ of a topological space $X$ if there is an open neighbourhood $U$ of $x$  such that the sets $f^{-n}(U)$, $n\geq 0$ are pairwise disjoint.

\section{Existence of non-contractible periodic points for homeomorphisms of the open annulus}

In order to prove Theorem A, we prove a recurrence type theorem in the lifted dynamics of a homeomorphism of the open annulus that is isotopic to the identity. Let $f$ be a homeomorphism of $\A$ which is isotopic to the identity and let $I'$ be an identity isotopy of $f$. A periodic point $z\in\A$ of period $q\in\N$ is said \textit{contractible (with respect to the isotopy $I'$)} if the loop $I'^q(z)$ is homotopically trivial in $\A$, otherwise it is said \textit{non-contractible (with respect to the isotopy $I'$)}. In this section we examine some conditions that ensure the existence of non-contractible periodic point of arbitrarily high period. We have the following result.

\begin{theob}\label{existencecontractiblepoints}
  Let $f$ be a homeomorphism of $\A$ which is isotopic to the identity and preserves a Borel probability measure of full support. Let $\check{f}$ be a lift of $f$ to $\Aa$. Suppose that $\check{f}$ has fixed points. Then one of the following alternatives must hold:
  \begin{itemize}
    \item [(1)] there exists an integer $q\geq 1$ such that for every irreducible rational number $r/s\in (0,1/q]$ the map $\check{z}\mapsto \check{f}^s(\check{z})+(r,0)$ has a fixed point or for every irreducible rational number $r/s\in (0,1/q]$ the map $\check{z}\mapsto \check{f}^s(\check{z})-(r,0)$ has a fixed point.

    \item [(2)] for every recurrent point $z\in\A$ there exists an open topological disk $V$ containing $z$ such that: if $\check{V}$ is a lift of $V$ and $\check{z}$ is the lift of $z$ contained in $\check{V}$, then for every integer $n$ satisfying $f^n(z)\in V$ we have $ \check{f}^n(\check{z})\in \check{V}$. In particular $\check{f}$ is non-wandering.
    \end{itemize}
\end{theob}
%
%
%
%

\textit{Proof of Theorem B.} Since $f$ is a homeomorphism of $\A$ isotopic to the identity, for every lift $\check{g}$ of $f$ to $\Aa$ one can always find an identity isotopy of $f$ that lifts to a path in the space of homeomorphisms of $\Aa$ joining the identity and $\check{g}$. Therefore, let $I'$ be an identity isotopy of $f$, such that its lift to $\check{\A},$ denoted $\check{I}',$ is an identity isotopy of $\check{f}$. By Theorem \ref{existence maximal isotopy} one can find a maximal identity isotopy $I$ of $f$ larger than $I'$. It can be lifted to an isotopy $\check{I}$ with $\check{\dom}(I)=\check{\pi}^{-1}(\dom(I))$. This isotopy is a maximal identity isotopy of $\check{f}$ larger than $\check{I}'$. By Theorem \ref{existence transverse foliation} one can  find an oriented singular foliation $\mathcal{F}$ which is transverse to $I$, its lift to $\check{\dom}(I)$,
denoted by $\check{\mathcal{F}}$ is transverse to $\check{I}$. Theorem B is  a consequence of the following proposition which will be proved below.

\begin{prop}\label{existencecontractiblepoints}
 Suppose that $f$ has contractible fixed points (with respect to $I$) and that one of the following conditions is satisfied:
    \begin{itemize}
    \item[(i)] There exist a linearly admissible transverse loop $\Gamma$ with a $\mathcal{F}$-transverse self-intersection, a lift $\check{\gamma}$ of the natural lift of $\Gamma$ to $\Aa$ and a non-zero integer $j$ such that for every $t\in\R$ we have $\check{\gamma}(t+1)=\check{\gamma}(t)+(j,0)$.
    \item[(ii)] There exist a linearly admissible transverse loop $\Gamma$ with a $\mathcal{F}$-transverse self-intersection, a lift $\check{\gamma}$ of the natural lift of $\Gamma$ to $\Aa$ and a non-zero integer $j$ such that $\check{\gamma}$ is the natural lift of a loop $\check{\Gamma}$ and $\check{\Gamma}$ and $\check{\Gamma}+(j,0)$ have a $\check{\mathcal{F}}$-transverse intersection.
    \item[(iii)] There exist an admissible $\mathcal{F}$-bi-recurrent transverse path $\gamma$ which has no $\mathcal{F}$-transverse self-intersection, a lift $\check{\gamma}$ of $\gamma$ to $\Aa$, a leaf $\check{\phi}$ of $\check{\mathcal{F}}$ and a non-zero integer $j$ such that $\check{\gamma}$ crosses both $\check{\phi}$ and $\check{\phi}+(j,0)$.
  \end{itemize}
  Then there exists an integer $q\geq 1$ such that for every irreducible rational number $r/s\in (0,1/q]$ the map $\check{z} \mapsto  \check{f}^s(\check{z}) +(r,0)$ or $\check{z} \mapsto  \check{f}^s(\check{z}) -(r,0)$ has a fixed point. In particular $f$ has non-contractible periodic points of arbitrarily high period.
\end{prop}

\begin{rema}
  If above condition (ii) holds, we will prove that there exists an integer $q\geq 1$ such that for every irreducible rational number $r/s\in [-1/q,1/q]$ the map $\check{z} \mapsto  \check{f}^s(\check{z})+(r,0)$ has a fixed point.
\end{rema}

\subsection{Some conditions that ensure the existence of non-contractible periodic points}

In this subsection we examine some conditions that allow us to apply Proposition \ref{existencecontractiblepoints}, and so to ensure the existence of non-contractible periodic points of arbitrarily high period. Let $\mathcal{F}$ be an oriented singular foliation on $\A$. We recall some facts about $\mathcal{F}$-bi-recurrent transverse path on $\A$. Let $\fonc{\gamma}{J\subset \R}{\A}$ be a $\mathcal{F}$-bi-recurrent transverse path. The path $\gamma$ being bi-recurrent, one can find real numbers $a<b$ such that $\phi_{\gamma(a)}=\phi_{\gamma(b)}$. Replacing $\gamma$ by an equivalent transverse path, one can suppose that $\gamma(a)=\gamma(b)$. Let $\Gamma$ be the loop naturally defined by the closed path $\gamma|_{[a,b]}$. We know that every leaf that meet $\Gamma$ is wandering  (see \cite{lct} for more details) and consequently, if $t$ and $t'$ are sufficiently close, one has $\phi_{\Gamma(t)}\neq \phi_{\Gamma(t')}$. Moreover, because $\Gamma$ is positively transverse to $\mathcal{F}$, one cannot find an increasing sequence $(a_n)_{n\in\N}$ and a decreasing sequence $(b_n)_{n\in\N}$, such that $\phi_{\gamma(a_n)}=\phi_{\gamma(b_n)}$. So, there exist real numbers $a',b'$ with $a\leq a'<b'\leq b$ such that $t\mapsto \phi_{\gamma(t)}$ is injective on $[a',b')$ and satisfies $\phi_{\gamma(a')}=\phi_{\gamma(b')}$. Replacing $\gamma$ by an equivalent transverse path, one can suppose that $\gamma(a')=\gamma(b')$. Let $\Gamma'$ be the loop naturally defined by the closed path $\gamma|_{[a',b']}$. The set $U_{\Gamma'}=\bigcup_{t\in [a',b'] } \phi_{\gamma(t)}$ is an open annulus and $\Gamma'$ is a simple loop. As the path $\gamma$ is a $\mathcal{F}$-bi-recurrent transverse path we have the following result, whose proof is contained in the proof of Proposition 2 from \cite{lct}.

\begin{leem}[\cite{lct}]\label{lemmaproofProposition2}
  Suppose that there exists $t<a'$ such that $\gamma(t)\notin U_{\Gamma'}$. Then there exists $t'\in \R$ with $b'<t'$ such that $\gamma(t)$ and $\gamma(t')$ are in the same connected component of the complement of $U_{\Gamma'}$. Moreover $\gamma|_{[t,t']}$ has a $\mathcal{F}$-transverse self-intersection.
\end{leem}
\begin{proof}
  See the proof of Proposition 2 from \cite{lct}.
\end{proof}

In the sequel, we assume, as in the previous subsection, that $f$ is a homeomorphism of $\A$ that is isotopic to the identity and preserves a Borel probability measure of full support, and that $I$ is a maximal identity isotopy of $f$. We write $\check{I}$ for the lifted isotopy of $I$ and $\check{f}$ for the lift of $f$ associated to $I$. This isotopy is a maximal identity isotopy of $\check{f}$. We suppose that $\mathcal{F}$ is an oriented singular foliation with $\dom(\mathcal{F})=\dom(I)$ which is transverse to $I$ and we write $\check{\mathcal{F}}$ for its lift to $\check{\A}$, which is transverse to $\check{I}$.

\begin{leem}\label{lemma}
Let $\fonc{\gamma}{[a,b]}{\A}$ be an admissible $\mathcal{F}$-transverse path and let $\fonc{\check{\gamma}}{[a,b]}{\Aa}$ be a lift of $\gamma$ to $\Aa$. Suppose that there exists a non-zero integer $j$ such that $\check{\gamma}$ and $\check{\gamma}+(j,0)$ intersect $\check{\mathcal{F}}$-transversally.
Then condition (i) or (ii) of Proposition \ref{existencecontractiblepoints} is satisfied.
\end{leem}
\begin{proof}
  By density of bi-recurrent points of $f$ and Lemma \ref{lemma10LCT} we can suppose that $\gamma$ is equivalent to a subpath  of the whole transverse trajectory of a bi-recurrent point. Since this whole transverse trajectory has a $\mathcal{F}$-transverse self-intersection (by hypothesis), by Lemma \ref{lemma30LCT} there exists a linearly admissible transverse loop $\Gamma'$ with a $\mathcal{F}$-transverse self-intersection, such that $\gamma$ is equivalent  to subpaths of the natural lift of $\Gamma'$. We note that, if $\gamma'$ is the natural lift of $\Gamma'$ and $\check{\gamma'}$ is the lift of $\gamma'$ to $\Aa$, then either $\check{\gamma'}$ is periodic or there exists some non-zero integer $j$ such that $\check{\gamma'}(t+1)=\check{\gamma'}(t)+(j,0)$ for all $t$. In the first case, since $\check{\gamma}$ is a subpath of $\check{\gamma'}$, one deduces that $\Gamma'$ satisfies condition (ii) of Proposition \ref{existencecontractiblepoints}, and in the second case $\Gamma'$ satisfies condition (i) of Proposition \ref{existencecontractiblepoints}.
  \end{proof}

From Lemmas \ref{lemmaproofProposition2} and \ref{lemma} we deduce the following result.

\begin{leem}\label{lemmaintersectiontransverse1}
Let $\fonc{\gamma}{[a,b]}{\A}$ be an admissible $\mathcal{F}$-transverse path. Suppose that there are real numbers $a<a'<b'<b$ such that $\gamma(a')=\gamma(b')$ and $t\mapsto \phi_{\gamma(t)}$ is injective on $[a',b')$. Let $U_{\Gamma'}$ be the open annulus associated to the loop naturally defined by the closed path $\gamma|_{[a',b']}$. Suppose furthermore that $\gamma(a)$ and $\gamma(b)$ belong to the same connected component of the complement of $U_{\Gamma'} $. Suppose that $\fonc{\check{\gamma}}{[a,b]}{\Aa}$ is a lift of $\gamma$ to $\Aa$ that has no $\check{\mathcal{F}}$-transverse self-intersection. Then condition (i) or (ii) of Proposition \ref{existencecontractiblepoints} is satisfied.
\end{leem}
\begin{proof}
By density of bi-recurrent points of $f$ and Lemma \ref{lemma10LCT} we can suppose that $\gamma$ is equivalent to a subpath  of the whole transverse trajectory of a bi-recurrent point. Since $\gamma(a)$ and $\gamma(b)$ belong to the same connected component of the complement of $U_{\Gamma'}$ we know, by Lemma \ref{lemmaproofProposition2}, that the path $\gamma$ has a $\mathcal{F}$-transverse self-intersection. Since $\check{\gamma}$ has no $\check{\mathcal{F}}$-transverse self-intersection, there exists a non-zero integer $j$ such that $\check{\gamma}$ and $\check{\gamma}+(j,0)$ intersect $\check{\mathcal{F}}$-transversally. Hence from Lemma \ref{lemma}, we know that condition (i) or (ii) of Proposition \ref{existencecontractiblepoints} is satisfied.  This completes the proof of the lemma.
\end{proof}

\begin{leem}\label{corolemmaintersectiontransverse3}
   Let $\fonc{\gamma}{[a,b]}{\A}$ be an admissible $\mathcal{F}$-transverse path and let $\fonc{\check{\gamma}}{[a,b]}{\Aa}$ be a lift of $\gamma$. Suppose that there exist a leaf $\check{\phi}$ of $\check{\mathcal{F}}$ and three distinct integers $j_i$, $1\leq i\leq 3$, such that $\check{\gamma}$ crosses each $\check{\phi}+(j_i,0)$. Then one of conditions of Proposition \ref{existencecontractiblepoints} is satisfied.
\end{leem}
\begin{proof}
 By density of the set of bi-recurrent points of $f$ and Lemma \ref{lemma10LCT} we can suppose that $\gamma$ is equivalent to a subpath  of  $\gamma'$, the whole transverse trajectory of a bi-recurrent point. Since the lift $\check{\gamma}'$ of $\gamma'$ that contains a subpath equivalent to $\check{\gamma}$ crosses each $\check{\phi}+(j_i,0)$, $1\leq i\leq 3$, it is sufficient to considerer the case where $\gamma'$ has a $\mathcal{F}$-transverse self-intersection. Otherwise $\gamma'$ satisfies condition (iii) of Proposition \ref{existencecontractiblepoints}. Hence by Lemma \ref{lemma30LCT} there exists a linearly admissible transverse loop $\Gamma''$ with a $\mathcal{F}$-transverse self-intersection such that $\gamma$ is equivalent to subpaths of the natural lift of $\Gamma''$. Write $\gamma''$ for the natural lift of $\Gamma''$ and for $\check{\gamma}''$ the lift of $\gamma''$ that contains a subpath equivalent to $\check{\gamma}$. We can suppose that  $\check{\gamma}''$ is also a loop $\check{\Gamma}''$, otherwise $\Gamma''$ satisfies condition (i) of Proposition \ref{existencecontractiblepoints}. In this case $\check{\Gamma}''$ is a $\check{\mathcal{F}}$-recurrent transverse path that crosses each $\check{\phi}+(j_i,0)$, $1\leq i\leq 3$. Hence we can prove that there exist $i\neq i'$ such that $\check{\Gamma}''-(j_i,0)$ and $\check{\Gamma}''-(j_{i'},0)$ intersect $\check{\mathcal{F}}$-transversally (see the proof of Proposition 43 from \cite{lct}). This implies that $\check{\Gamma}''$ and $\check{\Gamma}''+(j,0)$ intersect $\check{\mathcal{F}}$-transversally, where $j=j_i-j_{i'}$. Hence $\Gamma''$ satisfies condition (ii) of Proposition \ref{existencecontractiblepoints}. This completes the proof.
\end{proof}

From the previous lemmas and Proposition \ref{propocle} we deduce the following corollaries.

\begin{coro}\label{corolemmaintersectiontransverse1}
   Let $\fonc{\gamma}{[a,b]}{\A}$ be an admissible $\mathcal{F}$-transverse path and let $\fonc{\check{\gamma}}{[a,b]}{\Aa}$ be a lift of $\gamma$. Suppose that there exists an integer $j$ with $\abs{j}\geq 2$ such that $ \phi_{\check{\gamma}(b)}= \phi_{\check{\gamma}(a)}+(j,0)$. Then one of conditions of Proposition \ref{existencecontractiblepoints} is satisfied.
\end{coro}
\begin{proof}
 We will write $\check{\phi}$ by the leaf $\phi_{\check{\gamma}(a)}$ of $\check{\mathcal{F}}$. By Lemma \ref{corollary24LCT}, there exists an admissible $\check{\mathcal{F}}$-transverse path $\fonc{\check{\gamma}'}{[a,b]}{\Aa}$ such that $\check{\gamma}'$ has no $\check{\mathcal{F}}$-transverse self-intersection and $\phi_{\check{\gamma}'(a)}= \check{\phi}$, $\phi_{\check{\gamma}'(b)}= \check{\phi}+(j,0)$. We will suppose that $j\geq 2$, the other case is proved similarly. By Lemma \ref{corolemmaintersectiontransverse3} we can suppose that for every  $i\in\{1,\cdots,j-1\}$ we have that $\check{\gamma}'$ does not meet $\check{\phi}+(i,0)$.  Let us fix  $i\in\{1,\cdots,j-1\}$ and consider
  $$  a'_0:=\max\left\{ t\in [a,b] : \check{\gamma}'(t)\in \bigcup_{n\in \N} (\check{\phi}+(i-n,0))  \right\},   $$
  $$  b'_0:=\min\left\{ t\in [a'_0,b] : \check{\gamma}'(t)\in \bigcup_{n\in \N} (\check{\phi}+(i+n,0))  \right\},   $$
and let $\check{\gamma}'_0=\check{\gamma}'|_{[a'_0,b'_0]}$. Note that, as $\check{\phi}$ is disjoint from $\check{\phi}+(1,0)$, we can apply Proposition \ref{propocle} with $\check{\phi}$ in place of $\delta$, and deduce that $\check{\gamma}'_0$ meets $\check{\gamma}'_0+(1,0)$. Therefore there exist $a'_0\le t_1<s_1\le b'_0$ such that $\phi_{\check{\gamma}'(t_1)}=\phi_{\check{\gamma}'(s_1)+(1,0)}$ and since, by construction of $a'_0$ and $b'_0$, $\phi_{\check{\gamma}'(t_1)}$ cannot be equal to $\check{\phi}$, we have that $a'_0<t_1$ and $s_1<b'_0$. Hence $\check{\gamma}'_0$ projects onto an admissible $\mathcal{F}$-transverse path $\gamma'_0$ for which there exist real numbers $t_1\le a''_0<b''_0\le s_1$ such that $\gamma'_0(a''_0)=\gamma'_0(b''_0)$ and $t\mapsto \phi_{\gamma'_0(t)}$ is injective on $[a''_0,b''_0)$. Let $U_{\Gamma''_0}$ be the open annulus associated to the loop naturally defined by the closed path $\gamma'_0|_{[a''_0,b''_0]}$.  We note that by construction $\gamma'(a'_0)$ and $\gamma'(b'_0)$ belong to the same leaf (the projection of $\check{\phi}$) and since $\check{\gamma}'|_{(a'_0,b'_0)}$ does not intersect the translates of $\check{\phi}$, $\phi_{\gamma'_0(a'_0)}$ does not belong to the annulus $U_{\Gamma''_0}$. We deduce the corollary from Lemma \ref{lemmaintersectiontransverse1}.
\end{proof}

We need a slightly stronger result than the previous corollary, which does not have restrictions on the initial and final points of $\gamma$.

\begin{coro}\label{corolemmaintersectiontransverse11}
   Let $\fonc{\gamma}{[a,b]}{\A}$ be an admissible $\mathcal{F}$-transverse path and let $\fonc{\check{\gamma}}{[a,b]}{\Aa}$ be a lift of $\gamma$. Suppose that there exist a leaf $\check{\phi}$ of $\check{\mathcal{F}}$ and an integer $j$ with $\abs{j}\geq 2$, such that $\check{\gamma}$ crosses both $\check{\phi}$ and $\check{\phi}+(j,0)$. Then one of conditions of Proposition \ref{existencecontractiblepoints} is satisfied.
\end{coro}
\begin{proof}
Since in the proof of Lemma \ref{corolemmaintersectiontransverse3}  we can obtain a linearly admissible transverse loop $\Gamma''$ with a $\mathcal{F}$-transverse self-intersection such that $\gamma$ is equivalent to subpaths of the natural lift $\gamma''$ of $\Gamma''$. Consider real numbers $a\le a'<b'\le b$ such that $\phi_{\check{\gamma}(a')}=\check{\phi}$ and $\phi_{\check{\gamma}(b')}=\check{\phi}+(j,0)$, and let $t_1<t_2$ be such that $\gamma''|_{[t_1,t_2]}$ is equivalent to $\gamma|_{[a',b']}$. Since $\Gamma''$ has a $\mathcal{F}$-transverse self-intersection, there exists some $t_2<t_3<t_4$ such that $\gamma''|_{[s_1,s_2]}$ has a $\mathcal{F}$-transverse self-intersection for every $s_1<t_3$ and every $s_2>t_4$. Since $\Gamma''$ is admissible, this implies by Lemma \ref{Proposition19LCT} that for every $s_1<t_3<t_4<s_2$ there exists $n\ge 1$ such that $\gamma''|_{[s_1,s_2]}$ is admissible of order $n$. Let $k$ be a positive integer sufficiently large such that $t_1+k>t_4$, and let $\check{\gamma''}$ be the lift of $\gamma''$ such that $\phi_{\check{\gamma}''(t_1)}=\check{\phi}$. Then there exists an integer $r$ such that $\phi_{\check{\gamma}''(t_2)}=\check{\phi}+(j,0),\, \phi_{\check{\gamma}''(t_1+k)}=\check{\phi}+(r,0)$ and $\phi_{\check{\gamma}''(t_2+k)}=\check{\phi}+(r+j,0)$. If $\abs{r+j}\ge 2$, we can apply the previous corollary to $\gamma''|_{[t_1,t_2+k]}$, and if not, then $\abs{r-j}\ge 2$ and we can apply the previous corollary to $\gamma''|_{[t_2,t_1+k]}$.
\end{proof}

\subsection{Proof of Proposition \ref{existencecontractiblepoints}}

In this subsection, we show Proposition \ref{existencecontractiblepoints}. In the sequel, we assume that $f$ is a homeomorphism of $\A$ that is isotopic to the identity and preserves a Borel probability measure of full support, and that $I$ is a maximal identity isotopy of $f$. We write $\check{I}$ for the lifted isotopy of $I$ and $\check{f}$ for the lift of $f$ associated to $I$. This isotopy is a maximal identity isotopy of $\check{f}$. We suppose that $\mathcal{F}$ is an oriented singular foliation with $\dom(\mathcal{F})=\dom(I)$ which is transverse to $I$ and we write $\check{\mathcal{F}}$ for its lift to $\check{\A}$, which is transverse to $\check{I}$.\\

\textit{Let us suppose first that there exist a linearly admissible transverse loop $\fonc{\Gamma}{\T{1}}{\A}$ with a $\mathcal{F}$-transverse self-intersection, a lift $\check{\gamma}$ of the natural lift of $\Gamma$ to $\Aa$, and a non-zero integer $j$ such that for every $t\in\R$ we have $\check{\gamma}(t+1)=\check{\gamma}(t)+(j,0)$}. Suppose that $\Gamma$ satisfies the condition $(Q_q)$ for some integer $q\geq 1$. By Proposition \ref{proposition26LCT} we have that for every rational number $r/s\in (0,1/q]$ written in an irreducible way, the loop $\Gamma^r$ is associated to a periodic orbit of period $s$. In particular, the map $\check{z}\mapsto \check{f}^q(\check{z})+(j,0)$ or $\check{z}\mapsto \check{f}^q(\check{z})-(j,0)$ has a fixed point. Since $\check{f}$ has fixed points and the rotation set of $\check{f}$ is an interval (Theorem \ref{rotationset is a interval}) the proof of Proposition \ref{existencecontractiblepoints} follows in the first case.\\

\textit{Now let us suppose that there exists a linearly admissible transverse loop $\fonc{\Gamma}{\T{1}}{\A}$ with a $\mathcal{F}$-transverse self-intersection, a lift $\check{\gamma}$ of the natural lift of $\Gamma$ to $\Aa$ and a non-zero integer $j$ such that $\check{\gamma}$ is the natural lift of a loop $\check{\Gamma}$ and $\check{\Gamma}$ and $\check{\Gamma}+(j,0)$ have a $\check{\mathcal{F}}$-transverse self-intersection}. Choose an integer $L$ sufficient large, such that $\check{\gamma}|_{[0,L]}$ has a $\check{\mathcal{F}}$-transverse intersection with $\check{\gamma}|_{[0,L]}+(j,0)$ at $\check{\gamma}(t)=\check{\gamma}(s)+(j,0)$ with $s<t$. The loop $\check{\Gamma}$ being linearly admissible, there exists an integer $q\geq 1$ such that $\check{\gamma}|_{[-L,2L]}$ is admissible of order $q$. Applying Lemma \ref{corollary22LCT} with $\check{\gamma}|_{[-L,2L]}+(ij,0)$ in place of $\gamma_i$ or with $\check{\gamma}|_{[-L,2L]}-(ij,0)$ in place of $\gamma_i$ yields that for every integer $n\geq 1$, the paths $$  \prod_{i=0}^{n-1} \left(\check{\gamma}|_{[s-L,t+L]}+(ij,0)\right) \quad \text{ and } \quad  \prod_{i=0}^{n-1} \left(\check{\gamma}|_{[t-L,s+L]}-(ij,0)\right), $$
are admissible of order $\le nq$, and both have $\check{\mathcal{F}}$-transverse self-intersection as both paths contain as a subpath $\check{\gamma}|_{[0,L]}$. Lemma \ref{Proposition19LCT} shows that both paths are therefore admissible of order $nq$. Therefore the paths $\check{\gamma}|_{[s-L,t+L]}$ and $\check{\gamma}|_{[t-L,s+L]}$ project onto closed paths of $\A$ and the two loops naturally defined have $\mathcal{F}$-transverse self-intersection and are linearly admissible. We conclude as in the first case. This completes the proof of Proposition \ref{existencecontractiblepoints} in the second case.\\

\textit{Finally, let us suppose that there exist an admissible $\mathcal{F}$-bi-recurrent path $\fonc{\gamma}{\R}{\A}$ which has no $\mathcal{F}$-transverse self-intersection, a lift $\check{\gamma}$ of $\gamma$ to $\Aa$, a leaf $\check{\phi}$ of $\check{\mathcal{F}}$ and a non-zero integer $j$ such that $\check{\gamma}$ crosses both $\check{\phi}$  and $\check{\phi}+(j,0)$}. By Proposition \ref{proposition2LCT} the path $\gamma$ is equivalent to the natural lift of a transverse simple loop $\Gamma$, denoted still $\gamma$. Consider the set $U_\Gamma=\cup_{t\in \R} \phi_{\gamma(t)}$ which is an essential open annulus in $\A$, because $\check{\gamma}$ crosses both $\check{\phi}$  and $\check{\phi}+(j,0)$. This implies that $\Gamma$ is an essential simple loop, and so $\check{\gamma}$ intersects $\check{\phi}+(k,0)$ for every integer $k$. Note that either for every $t$, $\check{\gamma}(t+1)=\check{\gamma}(t)+(1,0)$ or for every $t$ ,$\check{\gamma}(t+1)= \check{\gamma}(t)-(1,0)$. We assume the first case holds, the other case being similar. Since $\gamma$ is admissible, there exist $L,n>0$ such that $\check{\gamma}|_{[0,L]}$ intersects $\check{\phi}$, $\check{\phi}+(1,0)$ and $\check{\phi}+(2,0)$, and is admissible of order $\le n$. Assume first that there exists $x$ such that the whole transverse trajectory of $x$ has a $\mathcal{F}$-transverse self-intersection and also contains a subpath equivalent to $\gamma|_{[0,L]}$. Then, as in the proof of Lemma \ref{corolemmaintersectiontransverse3}, one deduces that either condition (i) or (ii) of Proposition \ref{existencecontractiblepoints} hold, and we are done. Therefore we can assume that every point whose whole transverse trajectory contains a subpath equivalent to $\gamma|_{[0,L]}$ does not have a $\mathcal{F}$-transverse self-intersection. In particular, if such a point is bi-recurrent, then its whole transverse trajectory is equivalent to $\gamma$.

Since $U_{\Gamma}$ is an essential annulus, its lift $\check{U}_{\Gamma}$ to $\Aa$ is a foliated connected open set, homeomorphic to $\R^{2}$ and invariant by integer translations. Furthermore, if $\check{\phi}_0$ is a leaf of $\check{\mathcal{F}}$ in $\check{U}_{\Gamma}$ then $\check{\phi}_0$ intersects $\check{\gamma}$ and it also is a line in $\check{U}_{\Gamma}$. As we are assuming that $\check{\gamma}(t+1)=\check{\gamma}(t)+(1,0)$ for every $t$, one deduces that $\check{\phi}_0+(k,0)$ belongs to the right of $\check{\phi}_0$ in $\check{U}_{\Gamma}$ if $k>0$ and it belongs to the left of $\check{\phi}_0$ in $\check{U}_{\Gamma}$ if $k<0$. Therefore, if $\fonc{\beta}{[0,1]}{U_{\Gamma}}$ is a closed path positively transversal to $\mathcal{F}$ and $\check{\beta}$ is a lift of $\beta$ to $\Aa$, then $\check{\beta}(1)=\check{\beta}(0)+(l,0)$ with $l$ a strictly positive integer.

Let $z$ be a point such that $I^{n}_{\mathcal{F}}(z)$ contains a subpath equivalent to $\gamma|_{[0,L]}$. One can find, by Lemma \ref{lemma10LCT}, a small open ball $W\subset U_{\Gamma}$, with radius $r<1/4$ and containing $z$, such that for all $x$ in $W$ the whole transverse trajectory of $x$ contains a subpath equivalent to $\gamma|_{[0,L]}$. Let us fix $\check{W}$ a lift of $W$ to $\Aa$.

\begin{leem}
Let $x$ be a bi-recurrent point in $W$, and let $\check{x}\in \check{W}$ be a lift of $x$. If $i>0$ is such that $f^{i}(x)\in W$, then there exists $k>0$ such that $\check{f}^{i}(\check{x})$ belongs to $\check{W}+(k,0)$.
\end{leem}
\begin{proof}
If, by contradiction, $\check{f}^{i}(\check{x})$ belongs to $\check{W}+(k,0)$ for some $k\le 0$ then $I^{i+2}_{\check{\mathcal{F}}}(\check{f}^{-1}(\check{x}))$ must contain a subpath that is equivalent to a simple transverse closed curve $\fonc{\beta}{[0,1]}{\Aa}$ with $\beta(0)=\beta(1)+(k,0)$. But since the whole transverse trajectory of $x$ is contained in $U_{\Gamma}$, this is impossible.
\end{proof}

Since $f$ preserves a Borel probability measure of full support, by the Poincar\'e's Recurrence Theorem one can find a subset $B$ of $W$ with positive measure such that every $x$ in $B$ is bi-recurrent, and also such that, for every $x$ in $B$, there exists a strictly increasing sequence $n_k$ of integers with $\liminf_{k\to\infty}k/n_k=a>0$ where $f^{n_k}(x)$ belongs to $W$ for all $k$. But this implies, by the previous lemma and a simple induction argument, that if $x$ belongs to $B$ and $\check{x}$ is a lift $x$ in $\check{W}$, then $\check{f}^{n_k}(\check{x})$ must belong to $\check{W}+(i_k,0)$ for some integer $i_k\ge k$. One deduces that there exists a subsequence $n_{k_l}$, such that
$$\lim_{l\to\infty}\frac{1}{n_{k_l}}\left(p_1(\check{f}^{n_{k_l}}(\check{x}))- p_1(\check{x})\right)= \lim_{l\to\infty}\frac{i_{k_l}}{n_{k_l}}=\rho\ge a$$
where the limit is taken in $\overline{\R}$. Therefore the rotation set of $f$ contains both $0$ and $\rho>0$. Since this set is an interval, Proposition \ref{existencecontractiblepoints} follows in this case from Theorem \ref{rotationset is a interval}. This completes the proof of Proposition \ref{existencecontractiblepoints}.

\subsection{End of the proof of Theorem B}

\begin{proof}[End of the proof of Theorem B]
 Suppose that Assertion (ii) does not hold. There exists a $f$-recurrent point $z_0\in \A$ such that for every neighborhood $V$ of $z_0$, $\check{V}$ lift of $V$, and $\check{z}_0$ the lift of $z_0$ contained in $\check{V}$ there exists an integer $n \geq 1$ such that $f^n(z_0)\in V$, but $\check{f}^n(\check{z}_0)\in \check{V}+(j,0)$ for some non-zero integer $j$. On the other hand, by Lemma \ref{lemma10LCT}, we can find a neighborhood $W$ of $z_0$ such that for every $z\in W$ the path $I^2_{\mathcal{F}}(f^{-1}(z))$ crosses the leaf $\phi_{z_0}$. Let $\check{W}$ be a lift of $W$ to $\check{\A}$ containing $\check{z}_0$. Therefore Corollary \ref{corolemmaintersectiontransverse11} permits to conclude that for every neighborhood $\check{V}\subset \check{W}$ of $\check{z}_0$ there is an integer $n\geq 1$ such that $f^n(z_0)\in V$, but $\check{f}^n(\check{z}_0)$ belongs to either $\check{V}+(1,0)$ or $\check{V}-(1,0)$. Therefore the future orbit of $\check{z}_0$ accumulates either in $\check{z}_0+(1,0)$ or in $\check{z}_0-(1,0)$. We assume the former holds, the other case being similar. Note that, since $\check{f}$ commutes with integer translations, the future orbit of $\check{z}_0+(1,0)$ must accumulate on $\check{z}_0+(2,0)$, and one concludes by continuity that the future orbit of $\check{z}_0$ also accumulates on $\check{z}_0+(2,0)$. Therefore there must exists integers $0<m_1<m_2$ such that $\check{f}^{m_1}(\check{z}_0)$ belongs to $\check{W}+(1,0)$ and $\check{f}^{m_2}(\check{z}_0)$ belongs to $\check{W}+(2,0)$. Hence the transverse path $I_{\check{\mathcal{F}}}^{m_2+1}(\check{z}_0)$ crosses each $\phi_{\check{z}_0}+(i,0)$, $0\leq i \leq 2$. Therefore Lemma \ref{corolemmaintersectiontransverse3} permits to conclude that one of conditions of Proposition \ref{existencecontractiblepoints} is satisfied. We conclude using Proposition \ref{existencecontractiblepoints} that Assertion (i) of Theorem B holds. This completes the proof of Theorem B.
\end{proof}

\section{Proof of Theorem A}

In this section we prove Theorem A. Let $f$ be a homeomorphism of $\A$ which is isotopic to the identity. We suppose that $f$ preserves a Borel probability measure of full support, and so the set of bi-recurrent points is dense in $\A$. Assume that the connected components of the set of fixed points of $f$ are all compact. Let $\check{f}$ be a lift of $f$ to $\Aa$. Assume that $\check{f}$ has fixed points and that there exists an open topological disk $U\subset \A$ such that the set of fixed points of $\check{f}$ projects into $U$. We will suppose that Case $(2)$ of Theorem A does not hold, i.e. considering $f^{-1}$ instead of $f$ (if necessary) we will suppose that
 \begin{itemize}
   \item[($H_1$)] there exist a compact set $K_0$ of $\A$, a sequence of points $\suii{z}{l}$ in $K_0$, and a sequence of integers $\suii{n}{l}$ which goes to $+\infty$ such that the sequence
   $(f^{n_l}(z_l))_{l\in\N}$ is in $K_0$ and
        $$ p_1(\check{f}^{n_l}(\check{z}_l))-p_1(\check{z}_l)\geq M_l, $$
   where $\check{z}_l\in \check{\pi}^{-1}(z_l)$ and the sequence $\suii{M}{l}$ tends to $+\infty$ as $l$ goes to $+\infty$.
\end{itemize}
We will also suppose that Item (2) of Theorem B holds, otherwise Case (1) of Theorem A is true.
\begin{itemize}
  \item[($H_2$)] the set of bi-recurrent points of $\check{f}$ is dense in $\check{\A}$. In particular $\check{f}$ has no wandering points.
\end{itemize}

Write $\check{I}'$ for the lifted isotopy and $\check{f}$ for the lift of $f$ associated to $I'$. By Theorem \ref{existence maximal isotopy} one can find a maximal identity isotopy $I=(f_t)_{t\in[0,1]}$ larger than $I'$. It can be lifted to an isotopy $\check{I}=(\check{f}_t)_{t\in [0,1]}$ with ${\dom}(\check{I})=\check{\pi}^{-1}(\dom(I))$. This isotopy is a maximal singular isotopy of $\check{f}$ larger than $\check{I}'$. By Theorem \ref{existence transverse foliation} one can  find an oriented singular foliation $\mathcal{F}$ which is transverse to $I$, its lift to $\dom(\check{I})$, denoted by $\check{\mathcal{F}}$ is transverse to $\check{I}$.

\begin{rema}\label{remark2}
  Under above hypotheses we want to prove that there exist an admissible $\check{\mathcal{F}}$-transverse path $\check{\gamma}$, a leaf $\check{\phi}$ of $\check{\mathcal{F}}$ and a non-zero integer $j$ with $\abs{j}\geq 2$, such that $\check{\gamma}$ crosses both $\check{\phi}$ and $\check{\phi}+(j,0)$. Therefore Corollary \ref{corolemmaintersectiontransverse11} permits us to conclude that one of the conditions of Proposition \ref{existencecontractiblepoints} is satisfied, and so Case (1) of Theorem A holds. In fact using Hypothesis $(H_2)$ we can ``extend'' $\check{\gamma}$ to the natural lift of a transverse loop $\check{\Gamma}$ such that $\check{\Gamma}$ and $\check{\Gamma}+(j,0)$ have $\check{\mathcal{F}}$-transverse intersection. Therefore we conclude the following (see remark following Proposition \ref{existencecontractiblepoints})
  \begin{itemize}
    \item[(1')] there exists an integer $q\geq 1$ such that for every irreducible rational number $r/s \in [-1/q,1/q]$ the map the map $\check{z} \mapsto \check{f}^s(\check{z})+(r,0)$ has a fixed point.
  \end{itemize}
\end{rema}

We start recalling that from our suppositions, the set of fixed points of $\check{f}$ is not empty and it projects into an open topological disk $U$ of $\A$. We have the following result.

\begin{leem}\label{lemmadiameterbounded}
  Let $\check{U}$ be a connected component of $\check{\pi}^{-1}(U)$. Let $\mathcal{A}$ be an essential compact annulus of $\A$ and let $\check{\mathcal{A}}$ be the lift of $\mathcal{A}$ to $\Aa$. Then the diameter of the set $\check{U}\cap \check{\mathcal{A}}\cap \sing(\check{\mathcal{F}})$ is bounded.
\end{leem}
\begin{proof}
Suppose by contradiction that it is unbounded, i.e. there exists a sequence $\sui{\check{x}}$ of singularities of $\check{\mathcal{F}}$ in
$\check{U}\cap \check{\mathcal{A}}$ such that the sequence $\left(p_1(\check{x}_n)\right)_{n\in\N}$ is unbounded. By compactness of
$\sing(\mathcal{F})\cap \mathcal{A}$ we can suppose that the sequence $\left(\check{\pi}(\check{x}_n)\right)_{n\in\N}$ converges to $x$ in
$\sing(\mathcal{F})\cap \mathcal{A}$. Since the open disk $U$ contains all singularities of $\mathcal{F}$ we have that $x$ is in $U$. Let $\check{x}$ be the lift of $x$ contained in $\check{U}$. Since $\check{U}$ is an open disk in $\Aa$, we can find a small bounded ball $D$ contained in $\check{U}$ and
containing $\check{x}$. Therefore there are a non-zero integer $j$ and a large enough integer $n$ such that $\check{x}_n+(j,0)$ is in $D\subset \check{U}$. This implies that $\check{x}_n$ is in $(\check{U}-(j,0))\cap \check{U}$ with $j\neq 0$. Since $\check{U}$ is arcwise-connected this contradicts the fact that $U=\check{\pi}(\check{U})$ is an open topological disk in the open annulus $\A$. This completes the proof of the lemma.
\end{proof}

\begin{prop}\label{proposition1theoremA}
 Let $\mathcal{A}$ be an essential compact annulus of $\A$ and let
  $\check{\mathcal{A}}$ be the lift of $\mathcal{A}$ to $\Aa$. Then there exists a constant $M_{\mathcal{A}}>0$ such that if $\check{z}$ and $\check{f}^n(\check{z})$ belong to $\check{\mathcal{A}}$ and
  $$ \abs{p_1(\check{f}^{n}(\check{z}))- p_1(\check{z})}\geq M_\mathcal{A}$$ then one of the following must hold:
  \begin{itemize}
    \item[(a)] there exist a leaf $\check{\phi}$ of $\check{\mathcal{F}}$ and three distinct integers $j_i$, $1\leq i \leq 3$, such that the transverse path $I_{\check{\mathcal{F}}}^n(\check{z})$ crosses each $\check{\phi}+(j_i,0)$, or
    \item[(b)] there exists $m\in \{1,\cdots, n-1\}$ such that $\check{f}^{m}(\check{z})$ does not lie in $\check{\mathcal{A}}$.
  \end{itemize}
\end{prop}
\begin{proof}

From the previous lemma, one can find an open neighborhood $V\subset U$ of $\sing(\mathcal{F})$ such that if $\check{V}$ is a lift of $V$ to $\Aa$, the diameter of $\check{V}\cap \check{\mathcal{A}}\cap \sing(\check{\mathcal{F}})$ is bounded and such that for every point
$\check{z}\in \check{\pi}^{-1}(V)$, the points $\check{z}$ and $\check{f}(\check{z})$ belong to the same connected component of $\check{\pi}^{-1}(U)$.
One knows that for every point $z\in \mathcal{A}\setminus V$, there exists a small open disk $O\subset
\dom(\mathcal{F})$ containing $z$ such that for every $z'\in O$ the path $I_\mathcal{F}^2(f^{-1}(z'))$ crosses the leaf $\phi_z$. By compactness of $\mathcal{A}\setminus V$, one can cover this set by a finite family $(O_{i})_{1\leq i \leq r}$ and so one can construct a partition $(X_{i})_{1\leq i \leq r}$ of $\mathcal{A}\setminus V$ such that for every $i\in \{1,\cdots ,r\}$ we have $X_{i} \subset O_{i}$. We have a unique partition $(\check{X}_{\alpha})_{\alpha \in A}$ of $\check{\mathcal{A}}$ such that, either $\check{X}_{\alpha}$ is contained in a connected component of $\check{\pi}^{-1}(U)$ and projects onto $V$, or there exists $i\in\{1,\cdots,r\}$ such that $\check{X}_{\alpha}$ is contained in a connected component of $\check{\pi}^{-1}(O_{i})$, and projects onto $X_{i}$. We
write $\alpha(\check{z})=\alpha$, if $\check{z}\in \check{X}_{\alpha}$. Let us define
  $$ M_\mathcal{A}^0=\max_{\check{z}\in \mathcal{\check{A}}} \abs{p_1(\check{f}(\check{z}))-p_1(\check{z})}\quad \text{ and } \quad M_\mathcal{A}^1=
  \max_{\alpha\in A} \diam(p_1(\check{X}_\alpha)).$$\\
Let assume Assertion (b) of the proposition does not hold, and let $\check{z}\in \check{\mathcal{A}}$, and $n\geq 1$ an integer such that $\check{z},\, \check{f}(\check{z}),\, \cdots, \, \check{f}^{n}(\check{z})$ are contained in $\check{\mathcal{A}}$ and define a sequence $n_0<n_1<\cdots<n_s$ in the following inductive way:
  $$ n_0=0,\; n_{l+1}=1+\sup\{k\in \{n_l,\cdots,n-1\}| \alpha(\check{f}^k(\check{z}))=\alpha(\check{f}^{n_l}(\check{z}))   \},\; n_s=n. $$
  We have the following facts:
\begin{itemize}
  \item[(i)] for every $l\in\{1,\cdots, s-1 \}$ at least one the sets $\check{X}_{\alpha(\check{f}^{n_l}(\check{z}))}$ and
  $\check{X}_{\alpha(\check{f}^{n_{l+1}}(\check{z}))}$ does not project on $V\cap \mathcal{A}$;
  \item[(ii)] we have that $$\abs{p_1(\check{f}^{n_l}(\check{z}))-p_1(\check{f}^{n_{l+1}}(\check{z}))}\leq M^0_{\mathcal{A}}+M^1_{\mathcal{A}}.$$
\end{itemize}

 If $$\abs{p_1(\check{f}^{n}(\check{z}))- p_1(\check{z})}\geq (6r+1)(M_\mathcal{A}^0+M_\mathcal{A}^1)$$
then by above Property (ii), $s\geq 6r+1$ and so, by above Property (i), there exist at least $3r$ sets
$\check{X}_{\alpha(\check{f}^{n_{l_k}}(\check{z}))}$ that do not project on $\mathcal{A}\cap V$. This implies that Assertion (a) holds, completing the proof of the proposition.
\end{proof}

Let $\mathcal{A}_0$ be an essential compact annulus containing the compact set $K_0$ provided by Hypothesis $(H_1)$. Since from our suppositions, the connected components of the set of fixed points of $f$ are all compact, we deduce that the union of $\mathcal{A}_0$ with all connected components of the fixed point set intersecting $\mathcal{A}_0$ is also compact.  Hence we can consider two essential simple loops $\gamma_N$ and $\gamma_S$ in $\A$ which do not contain fixed points of $f$, such that the essential compact annulus $\mathcal{A}'$ limited by the paths $\gamma_N$ and $\gamma_S$ contains $A_0$. Let $B=f^{-1}(\mathcal{A}')\cup \mathcal{A}'\cup f(\mathcal{A}')$, and note that any fixed point for $f$ that lies in $B$ must also lie in $\mathcal{A}'$, and since $\gamma_N$ and $\gamma_S$ are fixed point free, it must lie in the interior of $\mathcal{A}'$.

 In the following, we assume that for each point $\check z$ in $\check{\A}$, each integer $j\ge 2$ and each leaf $\check{\phi}$ of $\check{\mathcal{F}}$, the trajectory $I_{\check{\mathcal{F}}}^{\Z}(\check z)$ does not cross both $\check{\phi}$ and $\check{\phi}+(j,0)$, otherwise Theorem A follows from Corollary \ref{corolemmaintersectiontransverse11}. We deduce the following proposition.

\begin{prop}
  If Theorem A does not hold, then there exist points $a$ and $b$ that are not fixed by the lift $\check{f}$, a sequence $\suii{\check{w}}{l}$, two sequences of integers $\suii{p}{l}$ and $\suii{q}{l}$ such that $$\lim_{l\to +\infty} p_l=\lim_{l\to +\infty} q_l=+\infty, \quad \lim_{l\to +\infty} \check{w}_l=a,\, \quad \lim_{l\to +\infty} (\check{f}^{q_l}(\check{w}_l)-(p_l,0))=b.$$
\end{prop}
\begin{proof}
 We note that from previous proposition, we can suppose that for every large integer $l$, some element of the set $\{\check{z}_l,\cdots,\check{f}^{n_l}(\check{z}_l)\}$ does not lie in $\check{\mathcal{A}'}$. As $\check{z}_l$ and $\check{f}^{n_l}(\check{z}_l)$ belongs to $\check{\mathcal{A}'}$ there exist an integer $i_l$ such that for every $i\in\{0,\cdots, i_l-1\}$,  $\check{f}^{i}(\check{z}_l)$ belongs to $\check{\mathcal{A}'}$ but $\check{f}^{i_l}(\check{z}_l)$ does not belong to $\check{\mathcal{A}'}$ and an integer $j_l$ such that for every $j\in\{j_l+1,\cdots, n_l\}$,  $\check{f}^{j}(\check{z}_l)$ belongs to $\check{\mathcal{A}'}$ but $\check{f}^{j_l}(\check{z}_l)$ does not belong to $\check{\mathcal{A}'}$. Let $\check{w}_l=\check{f}^{i_l}(\check{z}_l)$ and $q_l=j_l-i_l$, so $\check{f}^{q_l}(\check{w}_l)=\check{f}^{j_l}(\check{z}_l)$. Since both the image and pre-image of $\mathcal{A}'$ lie in $B$, we deduce that $\check{w}_l$ and $\check{f}^{q_l}(\check{w}_l)$ lie in $B\setminus\inte{\mathcal{A}'}$. Moreover if $\mathcal{A}$ is a compact annulus containing $B$, from previous proposition we have that

 $$ \abs{p_1(\check{z}_l)- p_1(\check{w}_l)}\leq M_{\mathcal{A}} \text{ and }
 \abs{p_1(\check{f}^{n_l}(\check{z}_l))- p_1(\check{f}^{q_l}(\check{w}_l))}\leq M_{\mathcal{A}}, $$
 and so
 $$ p_1(\check{f}^{q_l}(\check{w}_l))-p_1(\check{w}_l)\geq M_l -2M_{\mathcal{A}}.$$
  The proposition follows from the compactness of $B\setminus\inte{\mathcal{A}'}$ which has no fixed point of $f$.
\end{proof}

We denote by $\phi_a$ the leaf of $\check{\mathcal{F}}$ that contains $a$, and by $D'$ the set of all point $w$ whose whole transverse trajectory crosses the leaf $\phi_a$.
\begin{leem}
  The set $D'$ is an $\check{f}$-invariant connected open set which is horizontally unbounded, that is $\diam(p_1(D'))=+\infty$.
\end{leem}
\begin{proof}
  Note that the set $D'$ is $\check{f}$-invariant (by definition) and open by Lemma \ref{lemma10LCT}. Therefore for every large integer $l$ the point $\check{f}^{q_l}(\check{w}_l)$ belongs to $D'$, and so $D'$ is horizontally unbounded. It remains to prove that $D'$ is connected. Indeed, let $\widetilde{\pi}:\widetilde{\dom(\check{\mathcal{F}})}\to \dom(\check{\mathcal{F}})$ the universal covering map of $\dom(\check{\mathcal{F}})$ and let $\widetilde{\phi_a}$ be a lift of $\phi_a$ to the universal covering to $\dom(\check{\mathcal{F}})$. We consider $\widetilde{D'}$ the set of all point $\widetilde{w}$ whose whole transverse trajectory crosses the leaf $\widetilde{\phi_a}$. Recall that $\widetilde{\phi_a}$ is a Brouwer line, and if $O$ is the connected component of the complement of $\widetilde{\phi_a}\cup  \widetilde{f}(\widetilde{\phi_a})$ whose closure contains $\widetilde{\phi_a}\cup  \widetilde{f}(\widetilde{\phi_a})$ then $\widetilde{D'}=\cup_{n\in \Z} \widetilde{f}^n(\overline{O})$. It follows that $\widetilde{D'}$ is a connected set, and as $D'=\widetilde{\pi}(\widetilde{D'})$ we conclude that $D'$ is a connected set.
\end{proof}

Note that if $D'$ intersects $D'+(j,0)$, where $j$ is an integer with $\abs{j}\geq 2$ then there exists a point $w$ whose whole $\check{\mathcal{F}}$-transverse trajectory crosses both leaves $\phi_a$ and $\phi_a+(j,0)$. It follows from Corollary \ref{corolemmaintersectiontransverse11} that Case (1) of Theorem A holds (see Remark \ref{remark2}) and we are done. Suppose now that $D'$ and $D'+(j,0)$ are disjoint for all integer $j$ with $\abs{j}\geq 2$. We claim that this cannot be possible. Indeed, let $D$ be the union of $D'$ and all bounded connected components of the complement of $D'$. It follows, using the previous lemma, that $D$ is an $\check{f}$-invariant horizontally unbounded open disk such that $D$ and $D+(j,0)$ are disjoint if $\abs{j}\ge 2$. Moreover as $\check f$ commute with each integer translation we have that $D+(j,0)$ is also an $\check{f}$-invariant open disk. As $b$ is not fixed by $\check{f}$, we can consider a closed topological disk $B$ containing $b$ such that $\check{f}(B)\cap B= \emptyset$. As the sequence $(\check{f}^{q_l}(\check{w}_l)-(p_l,0))_{l\in \N}$ converge to $b$, we have that for all sufficiently large $l$, $\check{f}^{q_l}(\check{w}_l)-(p_l,0)$ belongs to both $D-(p_l,0)$ and $B$. As a consequence, $B$ intersects $D-(p_l,0)$ for at least three different $p_l$ (see Figure \ref{fig:setDproofthereomA}). Therefore $B$ intersects three pairwise disjoints $\check{f}$-invariant topological open disk, so by Theorem \ref{a triple boundary lemma} $\check{f}$ must have wandering points. This contradicts Hypothesis $(H_2)$ completing the proof of Theorem A.

\begin{figure}[h!]
  \centering
    \includegraphics{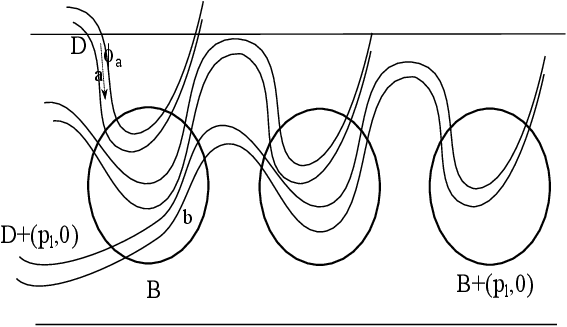}
  \caption{Set $D$ and its translations.}
  \label{fig:setDproofthereomA}
\end{figure}

\section{Examples}
Let us call a homeomorphism $\fonc{f}{\A}{\A}$ {\it irrotational} if its rotation set is reduced to an integer number. In this case, $f$ has a lift $\check{f}$ to $\Aa$ satisfying $\rot(\check{f})=\{0\}$, which we call the {\it irrotational lift of $f$}.

From Theorem A we know that, given an area-preserving irrotational homeomorphism $f$ of the open annulus $\A$ such that the connected components of the set of fixed point of $f$ are all compact and such that the set of fixed points of the irrotational lift, $\check{f}$, of $f$ to $\check{\A}$ projects into an open topological disk of $\A$, then for every compact set $K$ of $\A$ there exists a real constant $M>0$ such that for every $\check{z}\in \Aa$ and every integer $n\geq 1$ such that $\check{z}$ and $\check{f}^n(\check{z})$ belong to $\check{\pi}^{-1}(K)$ one has $$ \abs{p_1(\check{f}^n(\check{z}))-p_1(\check{z})}\leq M.$$

We will describe (in the following proposition) an example to show that in general one may not expect this bound to be independent of the compact set.

\begin{prop}\label{propexamfunb}
  There exists an irrotational homeomorphism $f_{bound}$ of $\A$ which preserves a Borel probability measure of full support satisfying:
  \begin{itemize}
    \item[(1)] the set of fixed points of the irrotational lift $\check{f}_{bound}$ of $f_{bound}$ projects into an open topological disk of $\A$;
    \item[(2)] the connected components of the set of fixed points of $f_{bound}$ are all compact;
    \item[(3)] for every real number $\eta>0$ and every integer $j\geq 1$ there exist a point $\check{z}\in\Aa$ and an integer $n\geq 1$ such that
    $$  \check{z}\in [0,\eta]\times [j,j+1] \quad \text{ and } \quad  \check{f}_{bound}^n(\check{z})\in [j,j+\eta]\times [j,j+1].$$
    \item[(4)] the homeomorphism $f_{bound}$ of $\A$ extends continuously to the semi-closed annulus $\A_{+\infty}:= \T{1}\times (-\infty,+\infty]$ as the identity on the circle $\T{1}\times \{+\infty\}$.
  \end{itemize}
\end{prop}

The example from previous proposition will allow us to describe another example showing that the hypothesis
``the connected components of the set of fixed point of $f$ are all compact'' is essential in the conclusion of Theorem A. In the following example none of the Cases of Theorem A holds.

\begin{prop}\label{propexamfnone}
  There is an irrotational homeomorphism $f_{none}$ of $\A$ which preserves a Borel probability measure of full support satisfying:
  \begin{itemize}
    \item[(1)] the set of fixed points of the irrotational lift $\check{f}_{none}$ of $f_{none}$ projects into an open topological disk of $\A$;
    \item[(2)] there exists a compact set $K_0$ of $\A$ such that for every real number $M>0$, there exist a point $\check{z}\in\Aa$ and an integer $n\geq 1$ such that $\check{z}$ and $\check{f}^n(\check{z})$ belong to $\check{\pi}^{-1}(K_0)$, and
        $$ \abs{p_1(\check{f}_{none}^n(\check{z}))-p_1(\check{z})}\geq M.$$
  \end{itemize}
\end{prop}

To obtain the example from Proposition \ref{propexamfunb}, it suffices to prove the following proposition.

\begin{prop}\label{propconstructionexample}
Let $x_0$, $x'_0$ and $x_1$, $x'_1$ be four points in $\T{1}$ and let $\epsilon>0$ be a real number. Then there exists a real number $\delta>0$ such that, given a positive integer $j$ and two irrotational diffeomorphisms $\fonc{g_0,g_1}{\T{1}}{\T{1}}$ which are different from the identity, $\delta$-close to the identity in the $C^1$-topology and such that $x_i$ and $x'_i$ are fixed points of $g_i$, $i\in\{0,1\}$, there is an area-preserving irrotational homeomorphism $\fonc{f}{\T{1}\times [0,1]}{\T{1}\times [0,1]}$ which is isotopic to the identity and satisfies:
\begin{itemize}
  \item[(a)] $f$ coincides with $g_0$ (resp. $g_1$) on the boundary component $\T{1}\times \{0\}$ (resp. $\T{1}\times \{1\}$);
  \item[(b-1)] there exists a path that joins $\T{1}\times \{0\}$ to $\T{1}\times \{1\}$, it is contained in the interior of $\T{1}\times [0,1]$, but the endpoints, and it does not intersect the set of fixed points of $f$;
  \item[(b-2)] there exists a loop in the interior of $\T{1}\times [0,1]$ which is not homotopic to zero, and it does not intersect the set of fixed points of $f$;
  \item[(c)] For every real number $\eta>0$ there exist a point $\check{z}\in \R\times [0,1]$ and an integer $n \geq 1$ such that
      $$\check{z}\in [0,\eta]\times [0,1]\quad \text{ and } \quad \check{f}^{n}(\check{z})\in [j,j+\eta]\times [0,1] .$$
  \item[(d)] $f$ is $\epsilon$-close to the identity in the $C^0$-topology.
\end{itemize}
\end{prop}

This proposition will be proved at the end of this section. In what follows, we will prove Propositions \ref{propexamfunb} and \ref{propexamfnone} assuming Proposition \ref{propconstructionexample}.

\subsection{Proof of Propositions \ref{propexamfunb} and \ref{propexamfnone}}

\begin{proof}[Proof of Proposition \ref{propexamfunb}]
For every $n\in\Z$ let $z_n:=(0,n)$ and $z'_n:=(1/2,n)$ be points in $\A$. Let us fix a sequence of positive real numbers $\left(\epsilon_n\right)_{n\in \Z}$ which converges to $0$ as $n$ to $+\infty$ and consider the sequence of positive real numbers $\left(\delta_n\right)_{n\in \Z}$ such that $\delta_n<\epsilon_n$ given by Proposition \ref{propconstructionexample}. Let us consider a sequence $\left(g_n\right)_{n\in \Z}$ of diffeomorphisms of the circle $\T{1}$ as in Proposition \ref{propconstructionexample} ($g_n$ fixes only $0$ and $1/2$) which are $\delta_n$-close to the identity to the identity in the $C^1$-topology, and a sequence of positive integers $\left(j_n\right)_{n\in \Z}$ which goes to $+\infty$ as $n$ to $+\infty$. For every integer $n$ we can define an area-preserving irrotational homeomorphism $f_n$ on the closed annulus $A_n=\T{1}\times [n,n+1]\subset \A$ which satisfies the properties formulated in Proposition \ref{propconstructionexample} and that the restriction of $f_n$ to $\T{1}\times  \{n\}$ is $g_n$ and the restriction of $f_n$ to $\T{1}\times \{n+1\}$ is $g_{n+1}$. Let us define the homeomorphism $f_{bound}$ of the open annulus $\A$ which coincides with $f_n$ on $A_n\subset \A$. We note that this homeomorphism preserves a Borel probability measure of full support. Moreover Properties (1) and (2) in Proposition \ref{propexamfunb} follows of Properties (b-1) and (b-2) in Proposition \ref{propconstructionexample} respectively. Moreover, Properties (3) and (4) in Proposition \ref{propexamfunb} follows of Property (c) in Proposition \ref{propconstructionexample} and Property (d) in Proposition \ref{propconstructionexample} and the chosen of the sequence $\left(\epsilon_n\right)_{n\in \Z}$ respectively. This completes the proof of the proposition.
\end{proof}

\begin{proof}[Proof of Propositions \ref{propexamfnone}]
Let us define the following equivalence relation on the semi-closed annulus  $\A_{+\infty}:=\T{1}\times (-\infty,+\infty]$.
$$   (x,y)\sim (x',y') \text {  if and only if } \quad \begin{cases}
  x=x' \text{ and } y=y';\\
  y=y'=+\infty,\, x\in [-1/4,1/4] \text{ and } x'=1-x.
\end{cases}     $$
Let $\A'_{+\infty}:=\A_{+\infty}/\sim$ be the quotient space of $\A_{+\infty}$. For $(x,y)\in\A_{+\infty}$, we write $[(x,y)]$ the equivalence class of $(x,y)$. One has that $\A'_{+\infty}$ is homeomorphic to $\A_{+\infty}$ and let $\fonc{\varphi}{\A'_{+\infty}}{\A_{+\infty}}$ be such a homeomorphism which acts as the identity on a neighborhood of the set $\{[(1/2,y)]: y\in \R\cup \{+\infty\}\}$, and such that the image by $\varphi$ of $\{[x, \infty]:0\le x\le 1/4\}$ is the segment $\{(0,t): 0\le t\le +\infty\}$.  Now, as $f_{bound}$ is the identity on the circle $\T{1}\times \{+\infty\}$ (Property (4) in Proposition \ref{propexamfunb}), it induces a homeomorphism $f'_{none}$ on $\A'_{+\infty}$ which acts as the identity on the segment $L_{+\infty}:= \{[(x,+\infty)]:x\in [0,1/4]\}$. We denote by $f_{none}$ the restriction to $\A$ of the homeomorphism  $ \varphi \, f'_{none}\,\varphi^{-1}$  defined  on the semi-closed annulus $\A_{+\infty}$. We note that $f_{none}$ is isotopic to the identity of $\A$. Let $K_0$ be a closed topological disk in $\A_{+\infty}$ that does not intersect the boundary of the semi-closed annulus and that contains $\varphi([(0,+\infty)])\in \varphi(L_{+\infty})$ in its interior. We note that there exist an integer $n_0$ and a real number $\eta_0>0$ such that $$\varphi(\{[(x,y)]: (x,y)\in [n_0,+\infty]\times [-\eta_0, \eta_0]\})\subset K_0.$$
(see Figure \ref{fig:examplefnone}).

\begin{center}
\begin{figure}[h!]
  \centering
    \includegraphics{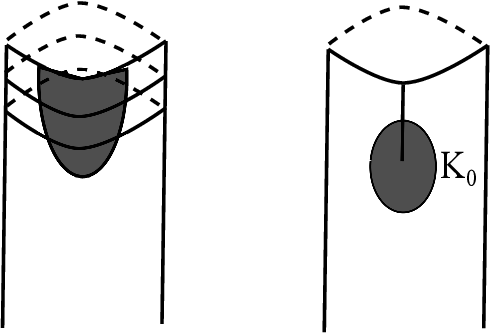}
  \caption{The set $K_0$.}
  \label{fig:examplefnone}
\end{figure}
\end{center}

Hence Property (2) in Proposition \ref{propexamfnone} follows of Property (3) in Proposition \ref{propexamfunb}. Moreover as $\varphi$ acts as the identity on a neighborhood of $\{1/2\}\times \R$ Property (1) in Proposition \ref{propexamfnone} follows of Property (1) in Proposition \ref{propexamfunb}. This completes the proof of Proposition \ref{propexamfnone}.
\end{proof}

\subsection{Proof of Proposition \ref{propconstructionexample}}

We use the following lemma (see \cite{kota}). It can be obtained by a straightforward adaptation of the proof of Proposition 2.2 from \cite{BCW}.

\begin{leem}
  For every real number $\epsilon>0$ there exists a real number $\delta>0$ such that, if $\fonc{g}{\T{1}}{\T{1}}$ is a diffeomorphism of the circle which is $\delta$-close to the identity in the $C^1$-topology, then there exists an area-preserving diffeomorphism  $\fonc{f}{\T{1}\times [0,1]}{\T{1}\times [0,1]}$ of the closed annulus which is $\epsilon$-close to the identity in the $C^1$-topology satisfying $f(x,1)=(g(x),1)$ and $f(x,0)=(x,0)$ for every $x\in \T{1}$.
\end{leem}

We deduce the following corollary.

\begin{coro}\label{coroconstfinD}
  For every real number $\epsilon>0$ there exists a real number $\delta>0$ such that, if $\fonc{g}{\s^{1}}{\s^{1}}$ is a diffeomorphism of the unit circle of the plane  which is $\delta$-close to the identity in the $C^1$-topology, then there exists an area-preserving homeomorphism $\fonc{f}{\overline{\D}}{\overline{\D}}$ of the closed unit disk which is $\epsilon$-close to the identity in the $C^1$-topology satisfying $f|_{ \s^{1}}=g$. Moreover if $z_0$ and $z_1$ are not fixed points of $g$, then there exists a path $\alpha$ in the interior of $\overline{\D}$, but the endpoints, that does not contain fixed points and that joins $z_0$ to $z_1$.
\end{coro}
\begin{proof}
  If $\fonc{g}{\s^{1}}{\s^{1}}$ is a diffeomorphism which is $\delta$-close to the identity, then applying the above lemma one can conclude that there exists an area-preserving diffeomorphism $f'$ as the above lemma. Furthermore, by performing an arbitrarily small area-preserving perturbation, we may assume that $f'$ has finitely many fixed points in $\T{1}\times(0,1)$. Collapsing the boundary component $\T{1}\times \{0\}$ to a point, we obtain an area-preserving homeomorphism $f$ of the closed unit disk $\overline{\D}$ of $\R^2$ which is $\epsilon$-close to the identity such that $f|_{\partial \overline{\D}}=g$. Since $f$ has finitely many fixed points in the interior of $\overline{\D}$, the existence of the path $\alpha$ follows.
\end{proof}

\begin{proof}[Proof of Proposition \ref{propconstructionexample}]
Our proof is an adaptation of the proof of the last claim from \cite{kota}. Let consider a topological flow $(\phi_t)_{t\in\R}$ on the closed annulus lifting to a flow $(\check{\phi}_t)_{t\in\R}$ of $\R\times [0,1]$ such that:
\begin{itemize}
  \item $\phi_t$ is area-preserving for every $t\in\R$;
  \item the square $\check{D}_0:=(0,1/2)\times (0,1)$ is $\check{\phi}_t$-invariant for every $t\in\R$;
  \item there are finitely many singularities and no essential closed ``connections'';
  \item there are not singularities on the boundary of $\check{D}_0$ but the vertices.
\end{itemize}
We note that by construction  for each $t\in\R$ there is a circle $\check{\phi}_t$-invariant, close enough of the boundary of $\check{D}_0$ such that the restriction of $\check{\phi}_t$ on this circle is transitive. Given an integer $j\geq 1$, let us consider the homeomorphism $H:\R\times [0,1] \to \R\times [0,1]$ defined as  $$ H(\check{x},y):= (\check{x},y)+ ((j+1)\sin (2\pi y),0). $$
We write $T: (\check{x},y)\mapsto (\check{x}+1,y)$. Note that $T H=H T$ and that the projection of the topological disk $\check{D}=H(\check{D}_0)$ onto the first coordinate has diameter greater than $j$. For every $t\in\R$, consider $$  \check{f}_t := H \check{\phi}_t H^{-1}. $$
It is easy to check that $T \check{f}_t=\check{f}_t T$, $\check{f}_t(\check{D})=\check{D}$, and that for every integer $n$,
$\check{f}_t^n=  H\check{\phi}_t^n H^{-1}$. Moreover given a real number $\eta>0$, both sets
$$\inte H^{-1}([0,\eta]\times [0,1])\quad \text{ and } \quad \inte H^{-1}([j,j+\eta]\times [0,1])$$ contain points in the boundary of $\check{D}$. Therefore, for every $t\in\R$ one can find as above a $\check{\phi}_t$-invariant and transitive circle intersecting both sets. This implies that there exists a point $\check{z}$ and an integer $n\geq 1$ such that $$\check{z} \in [0,\eta]\times [0,1]  \quad  \text{ and } \quad \check{f}_t^{n}(\check{z})\in [j,j+\eta]\times [0,1].$$

For $i\in \{-1,0,1\}$, let $A_i=\T{1}\times [i,i+1]$. Given $x_0$, $x'_0$, $x_1$ and $x'_1$ four points in $\T{1}$ let $z_{-1}= (x_0,-1)$, $z'_{-1}= (x'_0,-1)$; $z_{0}= (0,0)$, $z'_{0}= (1/2,0)$; $z_{1}= (0,1)$, $z'_{1}= (1/2,1)$; $z_{2}= (x_1,2)$, $z'_{2}= (x'_1,2)$. For each $i\in \{-1,1\}$ consider two disjoint segments $\alpha_{i}$ and $\alpha'_{i}$ joining $z_i$ to $z_{i+1}$ and $z'_i$ and $z'_{i+1}$ respectively, where both segments are contained in the interior of $A_i$ but for the endpoints.  These two segments divide the interior of the closed annulus $A_i$ in two open topological disks $D_i$ and $D'_i$ whose closure are closed disks. Given a real number $\epsilon>0$, by Corollary \ref{coroconstfinD} applied on each closed disk $\overline{D_i}$, $\overline{D'_i}$, $i\in \{-1,1\}$ we obtain a real number $0<\delta<\epsilon$. Define $f$ on $A_0$ as $f=f_t$, where $t$ is chosen small enough so as to guarantee that $f_t$ is $\delta$-close to the identity in the $C^1$-topology. For each $i\in \{-1,1\}$ let us choose two diffeomorphisms $\fonc{g_{\alpha_i}}{\alpha_i}{\alpha_i}$ and $\fonc{g'_{\alpha_i}}{\alpha'_i}{\alpha'_i}$ without fixed point but the endpoints close enough to the identity such that the diffeomorphisms induced on $\partial D_i$ and $\partial D'_i$ by  $g_0$, $g_1$, $f$, $g_{\alpha_i}$ and $g'_{\alpha_i}$ (and their inverses) are $\delta$-close to the identity. We can apply Corollary \ref{coroconstfinD} on each disk $\overline{D_i}$ and $\overline{D'_i}$ to obtain an area-preserving homeomorphism $f_{i}$ and $f'_i$ of $\overline{D_i}$ and $\overline{D'_i}$ respectively which are $\epsilon$-close to the identity. We consider now the homeomorphism of $\T{1}\times [-1,2]$ which coincides with $f$ on $A_0$ and with $f_i$ and $f'_i$ on $\overline{D_i}$ and $\overline{D'_i}$, $i\in \{-1,1\}$, respectively. Rescaling the annulus vertically we obtain a homeomorphism which satisfies the conditions of Proposition \ref{propconstructionexample}. This completes the proof of the proposition.
\end{proof}


  \textsc{Jonathan Conejeros. Instituto de Mat\'ematica e Estat\'{\i}stica, Universidade de S\~ao Paulo, Rua de Mat\~ao 1010, Cidade Universit\'aria, 05508-090 S\~ao Paulo, SP, Brazil}\\
  \textit{E-mail address:}  \texttt{jonathan@ime.usp.br}\\

  \textsc{F\'abio Armando Tal. Instituto de Mat\'ematica e Estat\'{\i}stica, Universidade de S\~ao Paulo, Rua de Mat\~ao 1010, Cidade Universit\'aria, 05508-090 S\~ao Paulo, SP, Brazil}\\
  \textit{E-mail address:} \texttt{fabiotal@ime.usp.br}

\end{document}